\documentclass[12pt]{amsart}

\usepackage{amsmath,amssymb,amscd,amsxtra,upref}
\usepackage{latexsym}
\usepackage[dvips]{graphics,epsfig}
\usepackage{enumerate}
\usepackage[colorlinks=true, pdfstartview=FitV, linkcolor=blue, citecolor=blue, urlcolor=blue]{hyperref}

\newtheorem{theorem}{Theorem}
\newtheorem{lemma}[theorem]{Lemma}

\newtheorem{corollary}[theorem]{Corollary}

\newtheorem{thm}{Theorem} 

\theoremstyle{remark}

\newtheorem{remark}{Remark}[section]

\newcommand{\s}{\mathcal{S}}
\newcommand{\dx}{\, dx}

\newcommand{\norm}[2]{\left\|#1\right\|_{#2}}

\newcommand{\dy}{\, dy}
\newcommand{\dz}{\, dz}

\newcommand{\derivs}{\abs{\partial_x^\alpha \partial_\xi^\beta \partial_\eta^\gamma \sigma(x,\xi,\eta)}}

\newcommand{\xxe}{(x,\xi,\eta)}
\newcommand{\upxe}{(1+\abs{\xi}+\abs{\eta})}

\newcommand{\xpe}{\abs{\xi}+\abs{\eta}}

\newcommand{\absx}{\abs{\xi}}
\newcommand{\abse}{\abs{\eta}}

\newcommand{\dxde}{\,d\xi \,d\eta}

\newcommand{\hphi}{\hat \phi}

\newcommand{\ei}[2]{e^{i #1 \cdot #2}}

\newcommand{\sym}{\sigma(x,\xi,\eta)}

\newcommand{\na}{\mathbb{N}}
\newcommand{\re}{\mathbb{R}}
\newcommand{\rn}{\mathbb{R}^n}
\newcommand{\rtn}{\mathbb{R}^{2n}}
\newcommand{\ent}{\mathbb{Z}}
\newcommand{\abs}[1]{\vert #1 \vert}
\newcommand{\FR}[2]{{\textstyle \frac{#1}{#2}}}

\newcommand{\fr}[2]{{\textstyle \frac{#1}{#2}}}

\newcommand{\tm}[1]{\text{\rm #1}}

\newcommand{\supp}{ {\rm supp} }

\newcommand{\bs}{BS_{\rho,\delta}^m}
\newcommand{\bsz}{BS_{\rho,\delta}^0}

\newcommand{\bsm}{BS_{\rho,\delta,K,N}^m}

\begin{document}




\thanks{ Partial NSF support under the following grants is acknowledged: third author  DMS 0901587; fourth author DMS 1101327; fifth author DMS 0800492 and DMS  1069015.}

\address{\'Arp\'ad B\'enyi, Department of Mathematics,
516 High St Western, Washington University, Bellingham, WA 98225,
USA.} \email{arpad.benyi@wwu.edu}

\address{Fr\'ed\'eric Bernicot, CNRS-Universit\'e Lille 1, Laboratoire de Math\'ematiques Paul Painlev\'e,
59655 Villeneuve d'Ascq Cedex (France).} \email{frederic.bernicot@math.univ-lille1.fr }

\address{Diego Maldonado, Department of Mathematics, 138 Cardwell Hall, Kansas State University,
Manhattan, KS 66506, USA.} \email{dmaldona@math.ksu.edu}

\address{Virginia Naibo, Department of Mathematics, 138 Cardwell Hall, Kansas State University,
Manhattan, KS 66506, USA.} \email{vnaibo@math.ksu.edu}

\address{Rodolfo H. Torres, Department of Mathematics, University of Kansas, Lawrence, KS
66045, USA.} \email{torres@math.ku.edu}

\title[Bilinear pseudodifferential operators]
{On the H\"ormander classes of bilinear pseudodifferential operators II}

\author[\'A. B\'enyi, F. Bernicot, D. Maldonado, V. Naibo, \and R. H. Torres]
{\'Arpad B\'enyi, Fr\'ed\'eric Bernicot, Diego Maldonado, Virginia Naibo, \and Rodolfo H. Torres}

\date{\today}

\begin{abstract}
Boundedness properties  for pseudodifferential operators with symbols in the bilinear H\"ormander classes of  sufficiently negative order are proved.  The results are obtained in the scale of  Lebesgue spaces and, in some cases,   end-point  estimates involving weak-type spaces and BMO are provided as well. From the Lebesgue space estimates, Sobolev ones are then easily obtained using functional calculus and interpolation. In addition, it is shown that, in contrast with the linear case,  operators associated with symbols of order zero may fail to be bounded on product of Lebesgue spaces.

 \end{abstract}

\maketitle

\section{Introduction}\label{intro}

In this article we continue the systematic study of the general H\"ormander classes of bilinear pseudodifferential operators $BS^{m}_{\rho, \delta}$ (see the next section for definitions) started in \cite{bmnt}.  While the work  in   \cite{bmnt} focussed mainly on basic properties related to the symbolic calculus of the bilinear pseudodifferential operators and some point-wise estimates for their kernels, the present work addresses boundedness properties on the full scale of Lebesgue spaces.  The general properties developed in \cite{bmnt} will become very useful in this current work and will allow us to provide a fairly complete range of results.

The literature on bilinear pseudodifferential operators continues to
grow and  \cite{bmnt} gives also a historical account and
motivations,  as well as numerous references in the subject. We
would like to reiterate here that most results so far have dealt
with the cases  $\rho=1$ and $\rho=0$. For the first value of $\rho$
the available boundedness  and unboundedness results, and  other
properties of the classes $BS^{0}_{1, \delta}$ are similar to the
ones in the linear situation. They are closely tied to the
(bilinear) Calder\'on-Zygmund theory, which was started by
Coiman-Meyer in the 70's (see e.g. \cite{CoMe} and the references
therein) and was further developed by Christ-Journ\'e \cite{CJ},
Kenig-Stein \cite{KS} and  Grafakos-Torres \cite{GT}.  See also
B\'enyi-Torres \cite{bt1} and Maldonado-Naibo \cite{MN}.
 The value of $\rho=0$, however,  produces some surprises and the
possible theory deviates from the linear situation. In particular
the famous Calder\'on-Villancourt theorem \cite{CaVa} does not hold
unless further properties on the symbols in $BS^{0}_{0, 0}$ are
imposed; see B\'enyi-Torres \cite{bt2} and
 Bernicot-Shrivastava \cite{bernicot-shrivastava}.

 One important contribution for other values of $\rho$, almost the exception so far, is the recent work
 of Michalowski-Rule-Staubach \cite{MRS}.  Since, for example, the class $BS^{0}_{0, 0}$ does not map
 $L^\infty \times L^2 \to L^2$, it was asked  in \cite{bmnt} (and some answers were provided) about
  results of the form  $X \times L^2 \to L^2$ with some functional space $X$ smaller than $L^\infty$
  and symbols in $BS^{0}_{\rho, \delta}$.  The question of whether the classes $BS^{0}_{\rho, \delta}$
  produce operators that are bounded on some product of Lebesgue spaces when $0 \le \delta < \rho$ was
  left unanswered in  \cite{bmnt} (recall the keystone result that the linear class $S^{0}_{\rho, \delta}$ is bounded on $L^2$,  as proved by H\"ormander \cite{Ho}). Likewise in \cite{MRS}
 the authors asked about which negative values of $m = m(\rho)$ produce classes $BS^{m}_{\rho, \delta}$  for which the corresponding bilinear pseudodifferential operators are bounded
 from $L^{p_1} \times L^{p_2} $ into $L^p$  with  $1/{p_1}+1/{p_2} =1/p$ and $1<p_1,p_2,p \leq \infty$.  Here, we will expand and improve some of the  results in  \cite{MRS} in several directions.

 First, we will show that it is very much relevant to look at negative values of $m$ when $\rho<1$ because operators with symbols in
 the classes $BS^{0}_{\rho, \delta}$ may fail to be bounded on any product of Lebesgue spaces. This is proved in Theorem~\ref{thm:main1} below, thus answering in the negative the question  left unanswered   in \cite{bmnt}.  Next we show in Theorem~\ref{thm:main2}  that the values of $m$ provided in \cite{MRS} can be taken much larger  (smaller in absolute value). We succeed in doing so using kernel estimates and the symbolic calculus from \cite{bmnt}, also used in \cite{MRS},  but adding arguments involving the complex interpolation of the classes $BS^{m}_{\rho, \delta}$. Moreover, bringing back the
 bilinear Calder\'on-Zygmund theory for sufficiently negative values of $m$ and using further interpolation arguments we also obtain results outside the {\it Banach triangle};  i.e., for  $1/p_1+1/p_2 =1/p,$  but $1/2<p<1$.  We also obtain appropriate weak-type end-point estimates at one end and a strong one at  another. This last is the bilinear analog of a result of C. Fefferman, which was also a keystone in the understanding of linear pseudodifferential operators.

  Fefferman \cite{fef} showed, in particular, that the linear classes $S^{-(1-\rho) \frac{n}{2}}_{\rho, 0}$, for $0<\rho<1$, map  $L^\infty \to BMO$. The natural conjecture then is that $BS^{-(1-\rho) n}_{\rho, 0}$ should map  $L^\infty \times L^\infty \to BMO$, since often the role of $n$ in the linear case is played by $2n$ in the bilinear setting.  We are able to prove this conjecture in Theorem~\ref{thm:main4} at least for  $0<\rho<1/2$. Though we use some ideas from \cite{fef},  new technical difficulties not present in the linear case  need to be overcome.
 In fact, Fefferman used the result of H\"ormander that operators with symbols in $S^{0}_{\rho, \delta}$ are bounded on $L^2$ but, as
 Theorem~\ref{thm:main1} establishes, the analogous result for bilinear operators is false. Instead we rely  on the
 $L^2 \times L^2 \to L^2$ boundedness of certain classes of symbols as proved in Theorem~\ref{thm:main3}.

 The article is organized as follows. In the next section we include the main definitions, some basic properties and the precise statements of the main theorems.  We also provide some further motivation and applications.
 The subsequence sections, Sections \ref{sec:main2}-\ref{sec:bssobolev}, contain the detailed proof of each of the main theorems in the order we list them, except that a series of technical lemmata used in the proof of Theorem~\ref{thm:main4} are postponed until Section~\ref{sec:prooflemmas}.  Section~\ref{sec:weights} contains some weighted versions of the results. Further remarks about the results and comparisons to other linear and bilinear ones are provided throughout the paper as well. Upper-case letters are used to label theorems corresponding to known results  while  single numbers are used for theorems, lemmas and corollaries that are  proved in this article.

Unless otherwise indicated, the underlying space for the functional classes used  will be Euclidean
space $\re^n$. In particular, $L^p$ will stand for $L^p(\re^n)$ and
 $W^{s,p}$  will stand for $W^{s,p}(\re^n),$  the Sobolev space of functions with ``$s$ derivatives'' in $L^p$. Their respective norms will be denoted $\norm{f}{L^p}$ and $\norm{f}{W^{s,p}}.$ Finally, $\mathcal{S}$ will indicate the Schwartz class on $\re^n.$

Throughout the symbol $\lesssim$ will be used in inequalities where constants are independent of its left and right hand sides.

\section{Main Results}

Let $\delta,\,\rho \geq 0$ and $m \in \re.$ In \cite{Ho}, H\"ormander introduced the class of symbols $S^m_{\rho,\delta}:$  $\sigma=\sigma(x,\xi),$ $x,\,\xi\in\re^n,$ belongs to $S^m_{\rho,\delta}$ if for all multi-indices $\alpha$ and $\beta$
\begin{equation*}
 \sup_{x,\,\xi \in\re^n}  |\partial_x^\alpha\partial_\xi^\beta \sigma (x,\xi)|  (1+|\xi|)^{-m-\delta
|\alpha|+\rho |\beta|}  < \infty.
\end{equation*}
For each symbol $\sigma$ there is an associated linear pseudodifferential operator $T_\sigma$  defined  by
\[
T_\sigma (f)(x)=\int_{\re^n} \sigma (x, \xi)\widehat
f(\xi)\,e^{ix\cdot \xi}\, d\xi,\qquad f\in \mathcal{S},
\]
where $\widehat{f}$ denotes the Fourier transform of $f.$

 The bilinear counterpart of $S^{m}_{\rho,\delta}$ is denoted $\bs.$ A bilinear symbol
$\sym,$ $x,\,\xi,\,\eta\in\re^n,$ belongs to the bilinear H\"ormander class $\bs$ if for all  multi-indices $\alpha,
\beta$ and $\gamma$,
\[
\sup_{x,\xi,\eta\in\re^n}  \derivs \upxe^{-m -\delta \abs{\alpha}+\rho(\abs{\beta}+\abs{\gamma})}<\infty.
\]
For $\sigma\in \bs$ and non-negative integers $K$ and
$N$ define
\[
\norm{\sigma}{K,N} := \mathop{\sup_{\abs{\alpha}\leq K}}_{ \abs{\beta},\abs{\gamma} \leq N}
 \sup_{x, \xi, \eta\in\re^n} {\derivs}\upxe^{-m - \delta \abs{\alpha}+ \rho(\abs{\beta}+\abs{\gamma})}.
 \]
Then the  family of norms $\{\|\cdot\|_{K,N}\}_{K,N\in \na_0}$ turns
$BS^m_{\rho,\delta}$ into a Fr\'echet space.

For $\sigma\in \bs$  we consider the bilinear pseudodifferential operator defined by
\begin{align*}
T_\sigma(f,g)(x)& := \int_{\re^n} \int_{\re^n}\sym \hat f(\xi) \hat g(\eta)\,\ei{x}{(\xi + \eta)} \,d\xi
\,d\eta, \quad f,\,g\in\mathcal{S}.
\end{align*}

We know proceed to state the new results  in this article.

\begin{theorem}\label{thm:main1} Let $0 \leq \rho < 1$, $0 \leq \delta \leq 1$, and
$1 \le  p, p_1, p_2 < \infty$ such that $\frac{1}{p} = \frac{1}{p_1}
+ \frac{1}{p_2}.$ There exist  symbols in $BS^0_{\rho,\delta}$ that
give rise to unbounded operators from $L^{p_1}\times L^{p_2}$ into
$L^p$.
\end{theorem}

As mentioned in the introduction, the result in Theorem~\ref{thm:main1}  is in  contrast with the
fact that linear pseudodifferential operators of order zero do
produce bounded operators on $L^2$. The case $\rho=\delta=0$ of Theorem~\ref{thm:main1} was
 proved by B\'enyi and Torres  in \cite{bt2}.

\begin{theorem}\label{thm:main2}
 Let $0\le \delta\le \rho\le 1,$ $\delta<1,$ $1\le p_1,p_2\le
\infty,$ $p$ given by $\frac{1}{p}=\frac{1}{p_1}+\frac{1}{p_2},$
\[
m<m(p_1,p_2):=n(\rho -1)\left(\max\{\FR{1}{2},\,\fr{1}{p_1},\, \fr{1}{p_2},\, 1-\fr{1}{p}\}+\max\{\fr{1}{p}-1,0\}\right),
\]
and $\sigma\in BS^m_{\rho,\delta}.$
\begin{enumerate}[(i)]
\item If $p\ge 1$ then  there exist $K,\,N\in\na_0$ such that
\[
\norm{T_\sigma(f,g)}{L^p}\lesssim \norm{\sigma}{K,N}\,\norm{f}{L^{p_1}}\norm{g}{L^{p_2}}
\]
for all $f\in L^{p_1}$ and $g\in L^{p_2}.$

\item If $0<\rho,$ $p<1,$  $p_1\neq 1$ and $p_2\neq 1$
then  there exist $K,\,N\in\na_0$ such that
\[
\norm{T_\sigma(f,g)}{L^p}\lesssim \norm{\sigma}{K,N}\,\norm{f}{L^{p_1}}\norm{g}{L^{p_2}}
\]
for all $f\in L^{p_1}$ and $g\in L^{p_2}.$

 \item If $0<\rho,$ $p<1$ and $p_1=1$ or $p_2=1$ then  then  there exist $K,\,N\in\na_0$  such that
\[
\norm{T_\sigma(f,g)}{L^{p,\infty}}\lesssim \norm{\sigma}{K,N}\,\norm{f}{L^{p_1}}\norm{g}{L^{p_2}}
\]
for all $f\in L^{p_1}$ and $g\in L^{p_2}.$
 \end{enumerate}
\end{theorem}

When $p\geq 1$ (Banach case), Theorem~\ref{thm:main2} improves  the results in  \cite[Theorem 5.5]{MRS} by Michalowski, Rule and Staubach which require $m<n(\rho -1)\max\{\FR{1}{2},(\fr{2}{p_1}-\fr{1}{2}), (\fr{2}{p_2}-\fr{1}{2}), (\fr{3}{2}-\fr{2}{p})\}.$
This improvement is based on the following facts:
\begin{enumerate}
 \item Bilinear pseudodifferential operators with symbols in the classes $BS^{m}_{\rho,\delta}$ with $m<n(\rho-1)$ (as opposed to $m<\frac{3}{2}n(\rho-1)$ used in \cite{MRS}) are bounded from $L^\infty\times L^\infty$ into $L^\infty,$ with norm bounded by the norm of the symbol (see also Remark~\ref{caseLinfty}).
   \item  Roughly speaking, the intermediate spaces in the complex interpolation of two bilinear H\"ormander classes  are other  bilinear H\"ormander classes.
\end{enumerate}
When $p<1$ (non-Banach case), the result of Theorem~\ref{thm:main2} relies on interpolation arguments using boundedness of operators in the Banach case and bilinear Calder\'on-Zygmund theory.

We remark that the operator $T_\sigma$ is a priori defined on $\mathcal{S}\times \mathcal{S}$. In Theorem~\ref{thm:main2}, $T_{\sigma}(f,g)$ for $f\in L^{p_1}$ and $g\in L^{p_2}$ denotes the ``value'' given by a bounded extension of the operator, which exists and is unique in the cases $p_1<\infty$ and $p_2<\infty,$ and is shown to exist when $p_1=\infty$ or $p_2=\infty.$

\begin{theorem} \label{thm:main3}
If  $\sigma(x,\xi,\eta),$ $x,\xi,\eta\in\re^n,$ is a bilinear symbol
such that
$$ C(\sigma):=\sup_{\genfrac{}{}{0pt}{}{|\beta|\le [\frac{n}{2}]+1}{|\alpha|\le 2(2n+1)}} \sup_{\xi,y\in\rn}  \|\partial_\xi^\alpha \partial_{y}^{\beta} \sigma(y,\xi-\cdot,\cdot) \|_{L^2} <\infty,$$
then $T_\sigma$ maps continuously $L^2 \times L^2$ into $L^2$ with
$$ \|T_\sigma\|_{L^2 \times L^2 \rightarrow L^2} \lesssim  C(\sigma).$$
\end{theorem}

\begin{theorem} \label{thm:main4}
 If $\sigma\in BS^{n(\rho-1)}_{\rho,0},$ $0 \le \rho <\frac{1}{2},$ then there exists $K,\,N\in\na_0$ such that
 \[
\norm{T_\sigma(f,g)}{BMO}\lesssim \norm{\sigma}{K,N}\,\norm{f}{L^{\infty}}\norm{g}{L^{\infty}}, \quad f,\,g\in\mathcal{S}.
\]
\end{theorem}

Theorem \ref{thm:main4}, which complements the endpoint $m=n(\rho-1)$ for $p_1=p_2=\infty$ in Theorem~\ref{thm:main2}, can be thought of as a bilinear counterpart (when $0\le \rho<\frac{1}{2}$ and $\delta=0$) to the following linear result proved by C. Fefferman in \cite{fef}.

\begin{thm}[Fefferman~\cite{fef}]\label{thmFef} If $\sigma$ is a symbol in the linear H\"ormander class
$S^{-\frac{n}{2}(1-\rho)}_{\rho, \delta}$ with $0 \leq \delta < \rho < 1$, then $T_\sigma$ maps $L^\infty$ continuously into $BMO$.
\end{thm}

The proof of Theorem~\ref{thmFef} uses the fact  that the linear class $S^{0}_{\rho,\delta}$, $0 < \delta < \rho \leq 1$, maps $L^2$ continuously into $L^2$. The bilinear counterpart of this result is false by Theorem~\ref{thm:main1}.
Our proof of Theorem~\ref{thm:main4} relies on Fefferman's ideas and the result given by Theorem~\ref{thm:main3}.

Next, we  present a result concerning boundedness properties of bilinear pseudodifferential operators  on Lebesgue spaces with  indices that satisfy the Sobolev scaling, as opposed to the  H\"older scaling employed in the  previous theorems.

\begin{theorem}\label{thm:bssobolev} Let $0 \leq \delta \leq 1,$ $0 < \rho \leq 1$, $s \in (0,2n)$, and $m_{s}:= 2n(\rho -1) - \rho s$. If $\sigma \in BS^{m}_{\rho,\delta},$ $m\le m_s,$
$1 < p_1, p_2 < \infty$, and $q>0$ is given by  $\frac{1}{q} = \frac{1}{p_1} + \frac{1}{p_1} - \frac{s}{n},$
then there exist  $K,\,N\in\na$ such that
\[
\norm{T_\sigma(f,g)}{L^q}\lesssim \norm{\sigma}{K,N}\,\norm{f}{L^{p_1}}\norm{g}{L^{p_2}}
\]
for all $f\in L^{p_1}$ and $g\in L^{p_2}.$
\end{theorem}

We end this section by briefly featuring some remarks, motivations and applications in the next three subsections.

\subsection{The operator norm, the number of derivatives, and complex interpolation of the classes of symbols}\label{lessderivatives}
Theorems~\ref{thm:main2} and \ref{thm:bssobolev} state that the operator norm of $T_\sigma$,  as a bounded operator from a product of Lebesgue spaces into another Lebesgue space, is controlled by $\|\sigma\|_{K,N}$ for some nonnegative integers $K$ and $N$.
Even though this is a consequence of the proof provided in each case, it can be shown to be  a necessary condition. More precisely,

\begin{lemma}\label{lemma:opnorm} Let $0< p\le \infty,$ $1\le p_1,p_2<\infty,$ $0\le \delta,\,\rho\le 1$ and
suppose $T_\sigma$ is bounded from $L^{p_1}\times L^{p_2}$ into
$L^p$ for all $\sigma\in BS^m_{\rho,\delta}.$ Then there exist
$K,\,N\in\na_0$  such that
\[
\|T_\sigma\|\lesssim\,\|\sigma\|_{K,N}\quad \text{for all } \sigma\in BS^m_{\rho,\delta}.
\]
\end{lemma}
Indeed, Lemma \ref{lemma:opnorm} is a consequence of the Closed Graph
Theorem. Consider in  $BS^m_{\rho,\delta}$ the topology induced by the family
of norms $\{\|\cdot\|_{K,N}\}_{K,N\in \na_0},$ as defined
above, which turns $BS^m_{\rho,\delta}$  into a Fr\'echet
space. If $T_\sigma$ is bounded from $L^{p_1}\times L^{p_2}$ into
$L^p$ for all $\sigma\in BS^m_{\rho,\delta}$ we can define the
linear transformation
\[
U:BS^m_{\rho,\delta}\to \mathcal{L}(L^{p_1}\times
L^{p_2},L^p),\qquad U(\sigma)=T_{\sigma},
\]
where  $\mathcal{L}(L^{p_1}\times L^{p_2},L^p)$ denotes the quasi-Banach
space  (Banach space if $p\ge 1$)  of all bilinear bounded operators from $L^{p_1}\times L^{p_2}$
into $L^p$ endowed with the operator quasi-norm (norm if $p\ge 1$). If
 $\{(\sigma_k,T_{\sigma_k})\}_{k\in\na}$ is a sequence in the graph of $U$ that converges to
$(\sigma,T),$  for some $\sigma\in BS^m_{\rho,\delta}$ and $T\in
\mathcal{L}(L^{p_1}\times L^{p_2},L^p),$ then it easily follows that
$T(f,g)=T_\sigma(f,g)$ for any $f,g\in \mathcal{S}(\re^n).$ Since
$T_\sigma$ and $T$ are bilinear bounded operators from $L^{p_1}\times
L^{p_2}$ into $L^p$, by density, we obtain that $T=T_\sigma$. Then
the graph of $U$ is closed and therefore, by the closed Graph
Theorem,  $U$ is continuous and the desired result follows.

\medskip

In regards to the number of derivates required for the symbols, we remark  that the following modified versions of the bilinear
H\"ormander classes can be considered: For $K,\,N\in\na_0,$
\[
\bsm:=\{\sigma(x,\xi,\eta)\in C^{K,N}(\re^{3n}):  \|\sigma\|_{K,N}<\infty \},
\]
where $C^{K,N}(\re^{3n})$ means derivatives up to order $K$ in $x$ and up to order $N$ in $\xi$ and $\eta.$
Then $\bsm$ is a Banach space with norm $\|\cdot\|_{K,N}$ that contains $\bs$ as a dense subset and therefore the results of Theorem~\ref{thm:main2},  \ref{thm:main4}, and \ref{thm:bssobolev} remain true if $BS^m_{\rho,\delta}$ is replaced with $BS^m_{\rho,\delta,K,N}$ for appropriate values of $K,\,N\in\na_0,$ possibly depending on $m,$ $\rho,$ and $\delta.$ We will not pursue in this paper the question regarding the minimum number of derivatives needed to achieve the results presented, though some estimates can be inferred from the proofs.

\medskip

We close this subsection with a result on the complex interpolation of the classes $BS^m_{\rho,\rho,N,N}$ which will be useful in the proof of Theorem~\ref{thm:main2}.

\begin{lemma} \label{interpolation} If $m_0,m_1\in\re,$ $0\le \rho<1$ and $m=\theta \,m_0+(1-\theta) \,m_1$ for some $\theta\in (0,1)$ then
\[
\left(BS^{m_0}_{\rho,\rho,N,N}, BS^{m_1}_{\rho,\rho,N,N}\right)_{[\theta]}=BS^{m}_{\rho,\rho,N,N}.
\]
\end{lemma}

Indeed, the lemma follows  using the same arguments as in  the work of  P\"aiv\"arinta-Somersalo \cite[Lemma 3.1]{PS}, where  the analogous
result for the linear H\"ormander classes is proved.

\subsection{Leibniz-type rules}
In terms of  applications of the bilinear $L^p$-theory for the class $BS^{m}_{\rho,\delta}$, the results in this paper  allow for enriched versions of the fractional Leibniz rule \begin{equation} \label{KaPoIneq}
\norm{fg}{W^{s,p}} \leq C \left(\norm{f}{W^{s,p_1}} \norm{g}{L^{p_2}} + \norm{f}{L^{p_1}} \norm{g}{W^{s,p_2}}  \right),
\end{equation}
where $s\ge 0$, $\frac{1}{p}=\frac{1}{p_1}+\frac{1}{p_2}$  and $1<p_1,\,p_2<\infty$ (see Kato-Ponce~\cite{KaPo}, Christ-Weinstein~\cite{CWein}, and Kenig-Ponce-Vega~\cite{KPVe}).

Inequalities of the type \eqref{KaPoIneq} for pseudodifferential operators $T_\sigma(f,g)$ instead of the product $fg$ ($\sigma\equiv 1$) can be  easily obtained following what is by now a well-known procedure that uses results going back to Coifman and Meyer and has become part of the folklore in the subject. The idea, as already used in \cite{KaPo},  is to (smoothly) split the symbol into frequency regions where the derivatives can be distributed among the functions.  See also  Semmes \cite{semmes} and Gulisashvili-Kon \cite{GuKo} where both homogeneous and inhomogeneous derivatives  were  considered in similar fashion.

Consider $\sigma \in BS^m_{\rho,\delta}$  and $\phi\in C^\infty(\re)$ such that $0 \leq \phi \leq 1$, $\supp(\phi) \subset [-2, 2]$ and $\phi(r) + \phi(1/r) = 1$ on $[0,\infty)$. For $s>0,$ the symbols $\sigma_1$ and $\sigma_2$ given by
\begin{align*}
\sigma_1(x,\xi, \eta)= \sigma(x,\xi, \eta) \phi\left(\frac{1+|\eta|^2}{1+|\xi|^2} \right) (1+|\xi|^2)^{-(m+s))/2},\\
\sigma_2(x,\xi, \eta)= \sigma(x,\xi, \eta) \phi\left(\frac{1+|\xi|^2}{1+|\eta|^2} \right) (1+|\eta|^2)^{-(m+s))/2},
\end{align*}
satisfy $\sigma_1, \sigma_2 \in BS^{-s}_{\rho,\delta}$, and the corresponding operators $T_\sigma$, $T_{\sigma_1}$, and $T_{\sigma_2}$ are related through
$$
T_\sigma(f,g) = T_{\sigma_1}(J^{m+s} f, g) + T_{\sigma_2}(f,J^{m+s} g),
$$
where $J^{m+s}$ denotes the linear Fourier multiplier with symbol $(1+|\cdot|^2)^{(m+s)/2}$. Thus, the boundedness properties on Lebesgue spaces of bilinear pseudodifferential operators given in Theorems \ref{thm:main2} and \ref{thm:bssobolev} imply
\begin{equation}\label{sigmaKaPo}
\norm{T_\sigma(f,g)}{L^p} \leq C \left(\norm{f}{W^{m+s,p_1}} \norm{g}{L^{p_2}} + \norm{f}{L^{q_1}} \norm{g}{W^{m+s,q_2}}  \right), \quad f, g \in \mathcal{S},
\end{equation}
for appropriate values of $p_1,\,p_2,\,q_1,\,q_2$ and $s$. We refer the reader to  Bernicot et al \cite{bmmn} for additional  Leibniz-type rules.

In the same spirit, using the functional rule
$$ \partial_{x_i} T_\sigma(f,g) = T_{\partial_{x_i} \sigma}(f,g) +
T_{\sigma}(\partial_{x_i} f,g)+T_{\sigma}(f,\partial_{x_i} g),$$ the
fact that $\sigma\in BS^m_{\rho,\delta}$ yields $\partial_{x_i}
\sigma \in BS^{m+\delta}_{\rho,\delta}$,  and bilinear complex
interpolation, Theorem~\ref{thm:main2} and
Theorem~\ref{thm:bssobolev} imply the following corollaries:

\begin{corollary} \label{cor:main1} Let $0\le \delta\le \rho\le 1,$
$\delta<1,$ $1\le p_1,p_2\le \infty,$ $p$ given by
$\frac{1}{p}=\frac{1}{p_1}+\frac{1}{p_2}$ and $m(p_1,p_2)$ as in
Theorem~\ref{thm:main2}. If $\sigma\in BS_{\rho,\delta}^m,$
$m<m(p_1,p_2)-k \delta$ for some nonnegative integer $k,$ and
$r\in[0,k],$ then there exists $K,\,N\in\na_0$ such that
$$\norm{T_\sigma(f,g)}{W^{r,p}}\lesssim
\norm{\sigma}{K,N}\norm{f}{W^{r,p_1}}\norm{g}{W^{r,p_2}}, $$ for all
$f\in W^{r,p_1}$ and $g\in W^{r,p_2}.$
\end{corollary}

\begin{corollary}\label{coro:bssobolev} Let $0 \leq \delta \leq 1,$ $0 < \rho \leq 1$,
 $s \in (0,2n)$,  $m_{s}= 2n(\rho -1) - \rho$ as in
 Theorem~\ref{thm:bssobolev},
$1 < p_1, p_2 < \infty$, and $q>0$ such that $\frac{1}{q} =
\frac{1}{p_1} + \frac{1}{p_1} - \frac{s}{n}.$ . If $\sigma \in
BS^{m}_{\rho,\delta},$ $m\le m_s-k\delta,$ for some nonnegative
integer $k,$ and $r\in[0,k],$ then there exists $K,\,N\in\na_0$ such
that
$$\norm{T_\sigma(f,g)}{W^{r,q}}\lesssim
\norm{\sigma}{K,N}\norm{f}{W^{r,p_1}}\norm{g}{W^{r,p_2}}, $$ for all
$f\in W^{r,p_1}$ and $g\in W^{r,p_2}.$
\end{corollary}

\subsection{Applications to the scattering of   PDEs}

Consider the system of partial differential equations for $u=u(t,x),$ $v=v(t,x)$, and $w=w(t,x),$ $t\in\re,$ $x\in\re^n,$
\begin{equation}
\left\{ \begin{array}{ll}  \partial_t u + a(D) u =vw, &u(0,x)=0,\label{eq:1}\\
 \partial_t v + b(D) v = 0, &v(0,x)=f(x),\\
  \partial_t w + c(D) w = 0, &w(0,x) = g(x).
 \end{array} \right.
\end{equation}
where $a(D),$  $b(D)$ and $c(D)$ are linear multipliers with symbols $a(\xi),$ $b(\xi)$ and  $c(\xi),$ $\xi\in\re^n,$ respectively.
Then, formally,
$$ v(t,x)= \int_{\re^n} e^{-t
b(\xi)} \widehat{f}(\xi)\, e^{ix\cdot \xi}\,d\xi,\quad w(t,x)= \int_{\re^n}  e^{-t
c(\eta)} \widehat{g}(\eta)\,e^{ix\cdot \eta}\,d\eta,$$
 and
$$ v(t,x)w(t,x) = \int_{\rtn} e^{-t
(b(\xi)+c(\eta))} \widehat{f}(\xi) \widehat{g}(\eta)\, e^{ix\cdot(\xi+\eta)} \,d\xi\,d\eta.$$
Another  formal computation then  yields
$$ u(t,x)= (e^{-ta(D)}F(t,\cdot))(x),$$
where
\begin{align*}
F(t,x) & = \int_0^t e^{sa(D)}(v(s,\cdot) w(s,\cdot))(x) ds \\
& = \int_{\rtn} \left(\int_0^t e^{s
(a(\xi+\eta)-b(\xi)-c(\eta))} ds \right) \widehat{f}(\xi)
\widehat{g}(\eta) \, e^{ix\cdot(\xi+\eta)} \,d\xi\,d\eta.
\end{align*}
Therefore, if the phase function $\lambda(\xi,\eta):=a(\xi+\eta)-b(\xi)-c(\eta)$
does not vanish,
$$ F(t,x) = T_{\frac{e^{t\lambda}-1}{\lambda}}(f,g)(x).$$

As a consequence, assuming that $\lambda<0, $ the solution $u$ of  (\ref{eq:1}) scatters in the Sobolev
space $W^{r,p}$ if
$$ \lim_{t\rightarrow \infty} T_{\frac{e^{t\lambda}-1}{\lambda}}(f,g) =
T_{-\lambda^{-1}}(f,g) \in W^{r,p}.$$
According to  Corollary \ref{cor:main1},  $T_{-\lambda^{-1}}$ is a bounded operator on Sobolev spaces if $-\lambda^{-1}$ belongs to
$BS^{m}_{\rho,\delta}$ for suitable exponents.

As an example consider  $b(D)=1-\Delta$  and $c(D)=|D|$.  Then for $a(D)=0$, we get
 $$-\lambda(\xi,\eta)^{-1}= (1+|\xi|^2+|\eta|)^{-1}$$
 and
 $$\lambda(\xi,\eta)^{-1}\varphi(\xi,\eta)\in BS^{-1}_{\frac{1}{2},0},$$
 for any smooth function $\varphi$ such that $\varphi=1$ away from the set $\{(\xi,\eta):\eta=0\}$. In the case that  $a(D)=\Delta$, we get
  $$-\lambda(\xi,\eta)^{-1}= (1+|\xi+\eta|^2+|\xi|^2+|\eta|)^{-1}$$
  and
  $$\lambda(\xi,\eta)^{-1}\varphi(\xi,\eta)  \in BS^{-2}_{1,0}.$$

When the phase function $\lambda$  vanishes, the situation is more difficult. We refer the reader to \cite{BG,BG1}, where a more precise study has been developed to obtain bilinear dispersive estimates (instead of scattering properties).

\section{Proof of Theorem \ref{thm:main1}}\label{sec:main2}

 As we will show,  Theorem~\ref{thm:main1} follows from the case  corresponding to $\rho=\delta=0,$ a scaling argument and Lemma~\ref{lemma:opnorm}.  We first need to recall the following result.

\begin{thm}[{B\'enyi-Torres~\cite[Proposition 1]{bt2}}] \label{bsCalVal}There exist $x$-independent symbols in $BS^0_{0,0}$ that give rise to unbounded operators from $L^{p_1}\times L^{p_2}$ into $L^p$ for $1\le p_1,\,p_2, p<\infty,$ $\frac{1}{p}=\frac{1}{p_1}+\frac{1}{p_2}.$
\end{thm}

\begin{proof}[Proof of Theorem~\ref{thm:main1}]

Fix $\delta,\rho,p_1,p_2,p$ as in the hypothesis. Suppose, on the contrary, that $T_{\sigma}$ is bounded
from $L^{p_1}\times L^{p_2}$ into $L^p$ for all $\sigma\in
BS^{0}_{\rho,\delta}.$

Consider an $x$-independent symbol $\sigma \in \bsz$ and, for  multi-indices $\beta, \gamma$, set
\[
C_{\beta, \gamma}(\sigma) :=\sup_{\xi,\eta\in\re^n}|\partial^\beta_\xi \partial^\gamma_\eta \sigma(\xi, \eta)| (1+|\xi|+|\eta|)^{\rho(|\beta|+|\gamma|)}.
\]
For $\lambda > 0$ define
$\sigma_\lambda(\xi, \eta) := \sigma(\lambda \xi, \lambda \eta),$ $\xi, \eta \in \re^n.$
Then, for  all multi-indices $\beta, \gamma$ and $0<\lambda<1$, we have
\begin{align*}
|\partial^\beta_\xi \partial^\gamma_\eta \sigma_\lambda(\xi, \eta)| & = \lambda^{|\beta|+|\gamma|}| \partial^\beta_\xi \partial^\gamma_\eta \sigma(\lambda \xi, \lambda \eta) | \\
& \leq  \lambda^{(1-\rho)(|\beta|+|\gamma|)} C_{\beta, \gamma}(\sigma) (1+|\xi|+|\eta|)^{-\rho(|\beta|+|\gamma|)},
\end{align*}
giving
\begin{equation}\label{Clambda}
C_{\beta, \gamma}(\sigma_\lambda) \leq  \lambda^{(1-\rho)(|\beta|+|\gamma|)} C_{\beta, \gamma}(\sigma).
\end{equation}
Let $f, \,g \in \mathcal{S}$ and define
$f_\lambda(x) := f\left(\frac{x}{\lambda}\right)$
and
$g_\lambda(x) := g\left(\frac{x}{\lambda}\right),$ $x \in \re^n.$
Then
\begin{align*}
T_\sigma(f,g)(x)& = \int_{\re^{2n}} \sigma(\xi, \eta) \hat{f}(\xi) \hat{g}(\eta) e^{i x \cdot (\xi + \eta)} \, d\xi d\eta\\
& = \int_{\re^{2n}}  \sigma\left(\lambda \frac{\xi}{\lambda}, \lambda \frac{\eta}{\lambda} \right) \hat{f}\left(\lambda \frac{\xi}{\lambda}\right) \hat{g}\left( \lambda \frac{\eta}{\lambda}\right) e^{i \lambda x \cdot \left(\frac{\xi}{\lambda} + \frac{\eta}{\lambda}\right)} \, d\xi d\eta\\
& = \int_{\re^{2n}} \sigma_\lambda(\xi, \eta) \widehat{f_\lambda}(\xi) \widehat{g_\lambda}(\eta) e^{i \lambda x \cdot(\xi + \eta)} \, d\xi d\eta\\
& = T_{\sigma_\lambda}(f_\lambda, g_\lambda)(\lambda x).
\end{align*}

Let $K,\,N\in\na_0$ be given by Lemma~\ref{lemma:opnorm} for the class $BS^0_{\rho,\delta}$ and, without loss of generality,   assume $K=N$. Then  using that  $\|\sigma_\lambda\|_{N,N}= \left(\sup\limits_{|\beta|,\,|\gamma| \leq N} C_{\beta, \gamma}(\sigma_\lambda)\right) ,$ $\frac{1}{p}=\frac{1}{p_1}+\frac{1}{p_2},$ and  \eqref{Clambda}, we obtain
\begin{align*}
\norm{T_\sigma(f,g)}{L^{p}} &= \norm{T_{\sigma_\lambda}(f_\lambda, g_\lambda)(\lambda \cdot)}{L^p} = \lambda^{-\frac{n}{p}} \norm{T_{\sigma_\lambda}(f_\lambda, g_\lambda)}{L^p}\\
& \lesssim  \lambda^{-\frac{n}{p}}  \left(\sup\limits_{|\beta|,\,|\gamma| \leq N} C_{\beta, \gamma}(\sigma_\lambda)\right)  \norm{f_\lambda}{L^{p_1}} \norm{g_\lambda}{L^{p_2}}\\
& =   \lambda^{-\frac{n}{p} + \frac{n}{p_1} + \frac{n}{p_2}}  \left(\sup\limits_{|\beta|,\,|\gamma| \leq N} C_{\beta, \gamma}(\sigma_\lambda)\right)  \norm{f}{L^{p_1}} \norm{g}{L^{p_2}}\\
& \lesssim \left(\sup\limits_{|\beta|,\,|\gamma| \leq N} \lambda^{(1-\rho)(|\beta|+|\gamma|)} C_{\beta, \gamma}(\sigma) \right)  \norm{f}{L^{p_1}} \norm{g}{L^{p_2}},
\end{align*}
and letting $\lambda \rightarrow 0$, it follows that
\begin{equation}\label{false}
\norm{T_\sigma(f,g)}{L^p} \lesssim\,C_{0,0}(\sigma) \norm{f}{L^{p_1}} \norm{g}{L^{p_2}} \quad f \in L^{p_1}, g \in L^{p_2}.
\end{equation}
However, \eqref{false} cannot be true since this contradicts Theorem \ref{bsCalVal}. Indeed, take $\sigma\in BS_{0,0}^0$ $x$-independent such that $T_\sigma$ is not bounded from $L^{p_1}\times L^{p_2}$
 into $L^p$ and let $\varphi$ be an infinitely differentiable  function in $\re^{2n}$ supported in $\abs{(\xi,\eta)}\le 2$ and equal to one on $\abs{(\xi,\eta)}\le 1.$   For each $\varepsilon>0,$ set $\sigma_{\varepsilon}(\xi,\eta):=\varphi(\varepsilon\,\xi,\varepsilon\,\eta)\sigma(\xi,\eta).$ Then $\sigma_\varepsilon\in BS_{\rho,\delta}^{0}(\re^n)$   and $C_{0,0}(\sigma_\varepsilon)\le C_{0,0}(\sigma)$ for all $\varepsilon>0.$ If \eqref{false} were true we would have
\[
\|T_{\sigma_\varepsilon}(f,g)\|_{L^p}\lesssim C_{0,0}(\sigma) \norm{f}{L^{p_1}} \norm{g}{L^{p_2}} \quad f,\, g \in \mathcal{S}, \quad \text{ for all } \varepsilon>0.
\]
As $\varepsilon\to 0,$  $T_{\sigma_\varepsilon}(f,g)\to T_{\sigma}(f,g)$ pointwise; this and Fatou Lemma yield
\[
\|T_{\sigma}(f,g)\|_{L^p}\lesssim C_{0,0}(\sigma) \norm{f}{L^{p_1}} \norm{g}{L^{p_2}} \quad f,\, g \in \mathcal{S},
\]
a contradiction.

\end{proof}

\section{Proof of Theorem \ref{thm:main2}}\label{sec:main1}

\subsection{Preliminary  results}

We will use the following results  in the proof of  Theorem~\ref{thm:main2}.

\begin{thm}[Symbolic calculus, B\'enyi-Maldonado-Naibo-Torres~\cite{bmnt}]\label{symbolic}
Assume that $0\leq \delta\le\rho\leq 1,$ $\delta<1,$ and $\sigma\in
BS_{\rho, \delta}^m$. Then, for $j=1,\,2,$
$T_\sigma^{*j}=T_{\sigma^{*j}},$ where $\sigma^{*j}\in BS_{\rho,
\delta}^m$.
Moreover, if $0\leq \delta<\rho\leq 1$ and $\sigma\in BS_{\rho, \delta}^m,$
then $\sigma^{*1}$ and $\sigma^{*2}$ have explicit asymptotic expansions.
\end{thm}

\begin{thm}[{Michalowski-Rule-Staubach~\cite[Theorem 5.5]{MRS}}] \label{MRS} Let $0\le \delta\le \rho\le 1,$ $\delta<1,$ $1\le p_1,p_2,p\le \infty,$ $\frac{1}{p}=\frac{1}{p_1}+\frac{1}{p_2}$ and
\[
m<n(\rho -1)\max\{\fr{1}{2},(\fr{2}{p_1}-\fr{1}{2}), (\fr{2}{p_2}-\fr{1}{2}), (\fr{3}{2}-\fr{2}{p})\}.
\]
If $\sigma\in BS_{\rho,\delta}^m,$ then there exist   $K,\,N\in\na_0$ such that
\[
\norm{T_\sigma(f,g)}{L^p}\lesssim \norm{\sigma}{K,N}\norm{f}{L^{p_1}}\norm{g}{L^{p_2}}.
\]
\end{thm}

Set   $\tilde{m}(p_1,p_2)=n(\rho -1)\max\{\fr{1}{2},(\fr{2}{p_1}-\fr{1}{2}), (\fr{2}{p_2}-\fr{1}{2}), (\fr{3}{2}-\fr{2}{p})\}$ and note that, when $p>1,$ we have  $m(p_1,p_2)=n(\rho -1)\max\{\FR{1}{2},\,\fr{1}{p_1},\, \fr{1}{p_2}, (1-\fr{1}{p})\}$.  Referring to Figure~\ref{fig1:main2},
  we then have that $m(p_1,p_2)=n(\rho -1)\frac{1}{p_2}$ and $\tilde{m}(p_1,p_2)=n(\rho -1)(\frac{2}{p_2}-\frac{1}{2})$ in region $I,$
$m(p_1,p_2)=n(\rho -1)\frac{1}{p_1}$ and $\tilde{m}(p_1,p_2)=n(\rho -1)(\frac{2}{p_1}-\frac{1}{2})$ in region $II,$
$m(p_1,p_2)=n(\rho -1)(1-\frac{1}{p})$ and $\tilde{m}(p_1,p_2)=n(\rho -1)(\frac{3}{2}-\frac{2}{p})$ in region $III,$ and
$m(p_1,p_2)=\tilde{m}(p_1,p_2)=n(\rho -1)\frac{1}{2}$ in region  $IV.$ Then $\tilde{m}<m$ in regions $I,$ $II$ and $III,$ and therefore the Banach case of Theorem~\ref{thm:main2} is an improvement on Theorem~\ref{MRS}.

In the non-Banach case ($p<1$),  we will use bilinear Calder\'on-Zygmund theory to get the boundedness results stated in Theorem~\ref{thm:main2}. Indeed, we have the following result:

\begin{theorem}[Bilinear Calder\'on-Zygmund operators]\label{thm:bscz} Let $0 \leq \delta \le \rho \le 1,$ $ \delta <
1,$ $0<\rho,$ and set $m_{cz}:=2n(\rho-1)$. If $\sigma \in
BS^{m}_{\rho,\delta}$ and $m< m_{cz},$ then $T_\sigma$ is a bilinear
Calder\'on-Zygmund operator. As a consequence, the following mapping
properties hold true for $1\le p_1,\,p_2\le \infty,$ $\frac{1}{2}\le
p<\infty,$ $\frac{1}{p}=\frac{1}{p_1}+\frac{1}{p_2}$:
\begin{enumerate}[(i)]
\item\label{first} if $1<p_1,\,p_2$, then there exist $K,\,N\in\na_0$  such that
\[
\norm{T_\sigma(f,g)}{L^p} \lesssim \norm{\sigma}{K,N}  \norm{f}{L^{p_1}}
\norm{g}{L^{p_2}},
\]
where $L^{p_1}$ or $L^{p_2}$ should be replaced by
$L^\infty_c$ (bounded functions with compact support) if
$p_1=\infty$ or $p_2=\infty,$ respectively;
\item if $p_1=1$ or $p_2=1$, then there exist $K,\,N\in\na$  such that
\[
\norm{T_\sigma(f,g)}{L^{p,\infty}} \lesssim \norm{\sigma}{K,N} \norm{f}{L^{p_1}}
\norm{g}{L^{p_2}},
\]
where $L^{p_1}$ or $L^{p_2}$ should be replaced by
$L^\infty_c$  if $p_1=\infty$ or $p_2=\infty,$ respectively;

\item there exist $K,\,N\in\na$  such that
\[
\norm{T_\sigma(f,g)}{BMO} \lesssim \norm{\sigma}{K,N}  \norm{f}{L^\infty}
\norm{g}{L^\infty}
\]
for $f,\,g\in L^\infty_c;$
\item weighted versions of the above inequalities (see Section~\ref{sec:weights}).

\end{enumerate}
\end{theorem}

The results of Theorem~\ref{thm:bscz}  are consequences of the following  estimates for the kernel of $T_\sigma$:

\begin{thm} \label{thm:kernelestimates} Let $\sigma \in BS^m_{\rho,\delta}, $ $0 < \rho \leq 1$, $0 \leq \delta < 1$, $m \in \re$, and denote by $\mathcal{K}(x,y,z)$
 the distributional kernel of the associated bilinear pseudodifferential operator $T_\sigma$. For $x, y, z \in \rn$, set
$$
S(x,y,z) = |x-y|+|x-z|+|y-z|.
$$
\begin{enumerate}[(i)]

\item\label{thm:kernelestimates-i} Given $\alpha, \beta, \gamma \in \na_0^n$, there exists $N_0 \in \na_0$ such that for each $l\geq N_0$,
$$
\sup_{(x,y,z): S(x,y,z) > 0} S(x,y,z)^l |D^\alpha_x D^\beta_y
D^\gamma_z \mathcal{K}(x,y,z)| < \infty.
$$
\item\label{thm:kernelestimates-ii} Suppose that $\sigma$ has compact support in $(\xi, \eta)$ uniformly in $x$. Then $\mathcal{K}$ is smooth, and given $\alpha, \beta, \gamma \in \na_0^n$ and
$N_0 \in \na_0$, there exists $C > 0$ such that for all $x, y, z \in \rn$ with $S(x,y,z)>0$
    $$
   |D^\alpha_x D^\beta_y D^\gamma_z \mathcal{K}(x,y,z)| \leq C (1 + S(x,y,z))^{-N_0}.
    $$
\item\label{thm:kernelestimates-iii} Suppose that $m + M + 2n < 0$ for some $M \in \na_0$. Then $\mathcal{K}$ is a bounded continuous function with bounded continuous derivatives of order $\leq M$.

\item\label{thm:kernelestimates-iv} Suppose that $m + M + 2n = 0$ for some $M \in \na_0$. Then there exists a constant $C > 0$ such that for all $x, y, z \in \rn$ with $S(x,y,z)>0$,
$$
\sup_{|\alpha + \beta + \gamma|=M}|D^\alpha_x D^\beta_y D^\gamma_z
\mathcal{K}(x,y,z)| \leq C |\log |S(x,y,z)||.
$$
\item\label{thm:kernelestimates-v} Suppose that $m + M + 2n > 0$ for some $M \in \na_0$. Then, given $\alpha, \beta, \gamma \in \na_0^n$, there exists a positive constant $C$ such that
for all $x, y, z \in \rn$ with $S(x,y,z) > 0$,
\begin{align*}
\sup\limits_{|\alpha + \beta + \gamma|=M} |\partial_x^\alpha
\partial_y^\beta \partial_z^\gamma \mathcal{K}(x,y,z)| \leq C S(x,y,z)^{-(m +
M + 2n)/\rho}.
\end{align*}
\item\label{thm:kernelestimates-vi} Suppose that $m + \varepsilon + 2n > 0$ for some $\varepsilon \in (0,1)$. Then, there exists a positive constant $C$ such that
for all $x, y, z, u \in \rn$ with $S(x,y,z) > 0$ and  $|u|\leq
S(x,y,z),$
\begin{align*}
 \lefteqn{|\mathcal{K}(x,y,z)-\mathcal{K}(x+u,y,z)|+|\mathcal{K}(x,y,z)-\mathcal{K}(x,y+u,z)|} & & \\
 & & + |\mathcal{K}(x,y,z)-\mathcal{K}(x,y,z+u)| \leq C |u|^\varepsilon S(x,y,z)^{-(m + \varepsilon + 2n)/\rho}.
\end{align*}
\end{enumerate}
All constants in the above inequalities depend linearly on $\norm{\sigma}{K,N}$ for some $K,\,N\in\na_0.$
\end{thm}

We refer the reader to \cite[Theorem 6]{bmnt} for the proofs of
items \eqref{thm:kernelestimates-i}-\eqref{thm:kernelestimates-v} in
Theorem~\ref{thm:kernelestimates}. Item
\eqref{thm:kernelestimates-vi} corresponds to the ``H\"older''
version of item \eqref{thm:kernelestimates-v}, its proof is
analogous and relies on estimates for linear kernels   as presented in Alvarez-Hounie
\cite[Theorem 1.1]{AH}.

\begin{proof}[Proof of Theorem~\ref{thm:bscz}]

It is enough to prove the result for  $\sigma \in
BS^m_{\rho,\delta}$  and $m$ such that   $2n(\rho-1)-t<m<2n(\rho-1)=m_{cz}$ for some small positive number $t.$ Denote by
$\mathcal{K}(x,y,z)$ the distributional kernel of the associated bilinear
pseudodifferential operator $T_\sigma$.
Using that $BS^m_{\rho,\delta}\subset BS^{m_{cz}}_{\rho,\delta},$ part \eqref{thm:kernelestimates-v} of Theorem~\ref{thm:kernelestimates} applied to $BS^{m_{cz}}_{\rho,\delta}$ yields, with constants depending linearly on $\norm{\sigma}{N,N}$ for some $N\in\na_0,$
\[
\abs{\mathcal{K}(x,y,z)}\lesssim \frac{1}{(|x-y|+|x-z|+|y-z|)^{2n}},
\]
while part \eqref{thm:kernelestimates-vi} gives, again with
constants depending linearly on $\norm{\sigma}{N,N}$ for some
$N\in\na_0,$
\begin{align*}
 \lefteqn{|\mathcal{K}(x,y,z)-\mathcal{K}(x+u,y,z)|+|\mathcal{K}(x,y,z)-\mathcal{K}(x,y+u,z)|} & & \\
 & &  + |\mathcal{K}(x,y,z)-\mathcal{K}(x,y,z+u)| \lesssim  \frac{|u|^{\varepsilon}}{(|x-y|+|x-z|+|y-z|)^{2n+\varepsilon}},
\end{align*}
 where   $|u|\leq |x-y|+|x-z|+|y-z|$
and  $\varepsilon\in (0,1)$ has been  chosen such that $(m+2n+\varepsilon)/\rho=2n+\epsilon$  (which is possible since $2n(\rho-1)-t<m<2n(\rho-1)$ for small enough $t>0$).
Moreover, since   $m<m_{cz}<n(\rho-1)/2,$
Theorem~\ref{MRS} yields that there exists $N\in\na_0$
such that $T_\sigma$  satisfies
\[
\norm{T_\sigma(f,g)}{L^1}\lesssim\norm{\sigma}{N,N}\norm{f}{L^2}\norm{g}{L^2}.
\]
We then conclude that $T_\sigma$ is a bilinear Calder\'on-Zygmund operator for which  the corresponding boundedness properties follow (see \cite{GT}).
\end{proof}

\subsection{Proof of Theorem~\ref{thm:main2}}

With these preliminary and technical results, we are now ready for the proof of our main result in this section.

\begin{proof}[Proof of Theorem~\ref{thm:main2}]

We first prove the theorem for $p_1=p_2=p=\infty,$ in which case $m(p_1,p_2)=n(\rho-1).$ Let $m<n(\rho-1),$ $0\le\delta\le \rho\le 1,$ $\delta<1.$
Let $\{\psi_j\}_{j\in\na_0}$ be a partition of unity on $\re^{2n},$
 \[\sum_{j=0}^\infty\psi_j(\xi,\eta)=1,\quad \xi,\eta\in\re^n,\]
 such that  $\psi_0$ is  supported in  $\{(\xi,\eta)\in \re^{2n}:\abs{(\xi,\eta)}\le 2\}$ and  $\psi_j(\xi,\eta)=\psi(2^{-j}\xi,2^{-j}\eta),$ for some $\psi\in \mathcal{C}^\infty_0(\re^{2n})$ supported in $\{(\xi,\eta)\in\re^{2n}:\abs{(\xi,\eta)}\sim 1\}$ for $j\in\na.$ We decompose the symbol $\sigma(x,\xi,\eta)$ as
 \[
 \sigma(x,\xi,\eta)=\sum_{j=0}^\infty \sigma_{j}(x,\xi,\eta),
 \]
where $\sigma_j(x,\xi,\eta):=\sigma(x,\xi,\eta)\psi_j(\xi,\eta).$ Then $\|\sigma_j\|_{0,N}\lesssim \|\sigma\|_{0,N}$ for all $N\in\na_0$ and,
by Lemma~\ref{compactsupp}  (see Section \ref{sec:main4}),
\[
\|T_{\sigma_j}(f,g)\|_\infty\lesssim \|\sigma\|_{0,2N}\,2^{j(m+n(1-\rho))} \|f\|_{L^\infty}\|g\|_{L^\infty},\quad j\in\na_0,\, N>n,\,K\in\na_0.
\]
Therefore
\begin{align*}
\|T_\sigma(f,g)\|_{L^\infty}\le\sum_{j=0}^\infty\|T_{\sigma_j}(f,g)\|_{\infty} & \lesssim \|\sigma\|_{0,2N} \sum_{j=0}^\infty 2^{j(m+n(1-\rho))} \|f\|_{L^\infty}\|g\|_{L^\infty}\\
&\lesssim \|\sigma\|_{0,2N}\,\|f\|_{L^\infty}\|g\|_{L^\infty},
\end{align*}
where we have used  that $m<n(\rho-1).$ This proves the theorem for $p_1=p_2=p=\infty.$ Note that the proof shows that there is an extension of $T_\sigma$ that is bounded from $L^\infty\times L^\infty$ into $L^\infty,$ mainly
\[
T_\sigma(f,g)=\sum_{j=1}^\infty T_{\sigma_j}(f,g)
\]
where
\[
T_{\sigma_j}(f,g)(x)=\int_{\re^{2n}} \mathcal{K}_j(x,x-y,x-z)f(y)g(z)\,dydz,
\]
with
\begin{equation*}
\mathcal{K}_j(x,y,z)=\int_{\re^{2n}}\sigma_j(x,\xi,\eta) \,e^{i\xi \cdot y}\, e^{i\eta\cdot  z}\,d\xi d\eta,\quad x,y,z\in \re^{n}.
\end{equation*}

\bigskip

We now proceed to prove the theorem in the general case. We recall
that the boundedness properties in Lebesgue spaces for operators
corresponding to the class $BS^0_{1,\delta}$ for $0\le \delta<1$ are
well-known (see introduction); therefore we will work with $\rho<1$.
Moreover, since $BS^m_{\rho,\delta}\subset BS^m_{\rho,\rho}$ for
$\delta\le \rho,$ we will assume $\delta=\rho,$ $0\le \rho<1.$
Define on $BS^m_{\rho,\rho}\times L^{p_1}\times L^{p_2}$ the
trilinear operator $T$  given by
\[
T(\sigma,f,g):=T_\sigma(f,g).
\]
In the following we will use the notation $T:BS^m_{\rho,\rho}\times X\times Y\to Z$  to express the fact that $T$ maps continuously from $BS^m_{\rho,\rho}\times X \times Y$ into $Z:$ there exists $N\in\na_0,$ possibly depending on $m$ and $\rho,$ such that
\[
\norm{T(\sigma,f,g)}{Z}\lesssim \norm{\sigma}{N,N}\norm{f}{X}\norm{g}{Y},
\]
for all $\sigma\in BS^m_{\rho,\rho},$ $f\in X,$ $g\in Y.$
\begin{figure}[ht] 
   \begin{center}
   \includegraphics[width=4in]{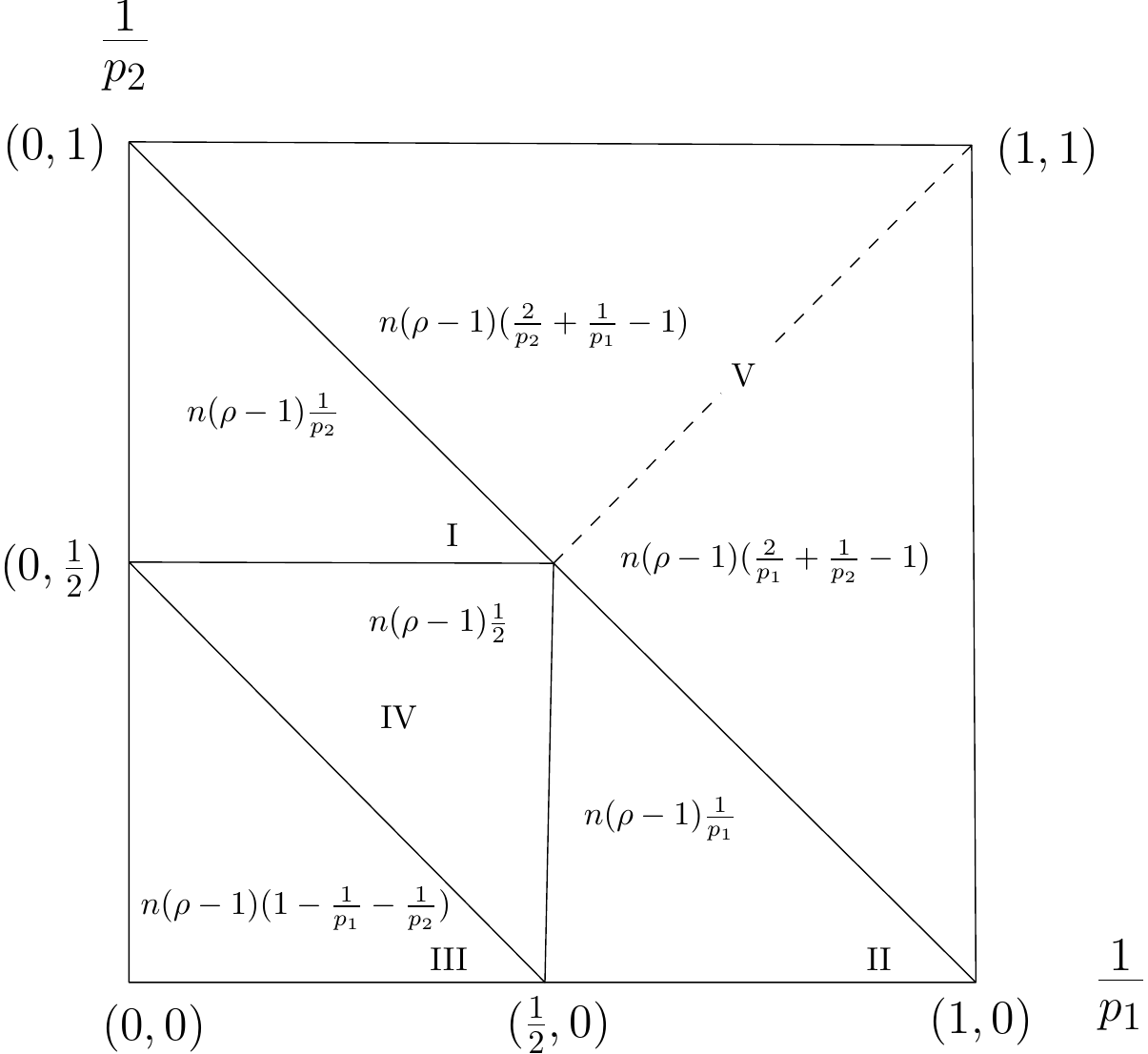}
\end{center}
   \caption{Value of $m(p_1,p_2)$ as given by Theorem~\ref{thm:main2}}
   \label{fig1:main2}
\end{figure}

\medskip

We first prove (i) (case $p>1$). The case $p_1=p_2=p=\infty$ proved above and
Theorem~\ref{symbolic} yield
\begin{enumerate}
\item[\tiny$\bullet$] $T: BS^m_{\rho,\rho}\times L^{\infty}\times L^{\infty}\to L^\infty$ for $m<n(\rho-1)$ (point $(0,0)$ in Figure~\ref{fig1:main2}),
\item[\tiny$\bullet$] $T: BS^m_{\rho,\rho}\times L^{1}\times L^{\infty}\to L^1$ for $m<n(\rho-1)$ (point $(1,0)$ in Figure~\ref{fig1:main2}),
\item[\tiny$\bullet$] $T: BS^m_{\rho,\rho}\times L^{\infty}\times L^{1}\to L^1$ for $m<n(\rho-1)$ (point $(0,1)$ in Figure~\ref{fig1:main2}).
\end{enumerate}
Moreover, by Theorem~\ref{MRS} we have
\begin{enumerate}
\item[\tiny$\bullet$] $T: BS^m_{\rho,\rho}\times L^{2}\times L^{2}\to L^1$ for $m<\frac{n}{2}(\rho-1)$ (point $(\fr{1}{2},\fr{1}{2})$ in Figure~\ref{fig1:main2}),
\item[\tiny$\bullet$] $T: BS^m_{\rho,\rho}\times L^{2}\times L^{\infty}\to L^2$ for $m<\frac{n}{2}(\rho-1)$ (point $(\fr{1}{2},0)$ in Figure~\ref{fig1:main2}),
\item[\tiny$\bullet$] $T: BS^m_{\rho,\rho}\times L^{\infty}\times L^{2}\to L^2$ for $m<\frac{n}{2}(\rho-1)$ (point $(0,\fr{1}{2})$ in Figure~\ref{fig1:main2}).
\end{enumerate}

We now recall the  following modified version of the bilinear
H\"ormander classes (see Section~\ref{lessderivatives}):
\[
BS^{m}_{\rho,\rho,N,N}:=\{\sigma(x,\xi,\eta)\in C^N(\re^{3n}): \|\sigma\|_{N,N}<\infty \}
\]
where  $N\in \na_0$  and, as
always,
\[
\norm{\sigma}{N,N} := \mathop{\sup_{\abs{\alpha}\leq N}}_{ \abs{\beta},\abs{\gamma} \leq N}
 \sup_{x, \xi, \eta\in\re^n} {\derivs}\upxe^{-m - \rho \abs{\alpha}+ \rho(\abs{\beta}+\abs{\gamma})}.
 \]

 Since  $ BS^{m}_{\rho,\rho}$ is dense in $ BS^{m}_{\rho,\rho,N,N},$ the above mentioned endpoint results also hold if $ BS^{m}_{\rho,\rho}$ is replaced with $ BS^{m}_{\rho,\rho,N,N}$ for
 large enough $N$ possibly depending on $\rho$ and $m.$   Lemma~\ref{interpolation} and
 trilinear complex interpolation (see the book of Bergh and L\"ofstr\"om \cite[Theorem 4.4.1]{BL}) then yield the thesis of the theorem for $p_1$
and $p_2$ such that $(\frac{1}{p_1},\frac{1}{p_2})$ is on the border
of the triangle with vertices $(0,0),$ $(0,1)$ and $(1,0)$.

The result for $p_1$ and $p_2$ such that
$(\frac{1}{p_1},\frac{1}{p_2})$ is in the interior of the triangle
follows by  bilinear complex interpolation since, as shown in Figure~\ref{fig1:main2},  $m(p_1,p_2)$ is constant
along horizontal segments in region $I,$ $m(p_1,p_2)$ is constant
along  vertical segments in region $II,$ $m(p_1,p_2)$ is constant
along diagonal segments in region $III$ and $m(p_1,p_2)$ is constant
in region $IV.$

We now prove (ii) and (iii) (case $p<1$). Here we have to assume $\rho>0.$ Theorem~\ref{thm:bscz}
yields
\[
T: BS^m_{\rho,\rho}\times L^{1}\times L^{1}\to L^{\frac{1}{2},\infty} \quad \textrm{for $m< 2n(\rho-1)$ (point $(1,1)$ in Figure~\ref{fig1:main2})},
\]
which together with the boundedness properties at the points $(1,0)$
and $(0,1)$ in Figure~\ref{fig1:main2} (as stated above), Lemma~\ref{interpolation},
and trilinear complex interpolation gives that
\[
T: BS^m_{\rho,\rho}\times L^{p_1}\times L^{p_2}\to L^{p,\infty},\quad  m< m(p_1,p_2),
\]
 for $(\frac{1}{p_1},\frac{1}{p_2})$ on the segments joining the
points $(0,1)$ to $(1,1),$  $(1,0)$ to $(1,1),$ and
$(\frac{1}{2},\frac{1}{2})$ to $(1,1),$ in Figure~\ref{fig1:main2}. This gives Part (iii).

\begin{figure}[h] 
   \begin{center}
   \includegraphics[width=2.7in]{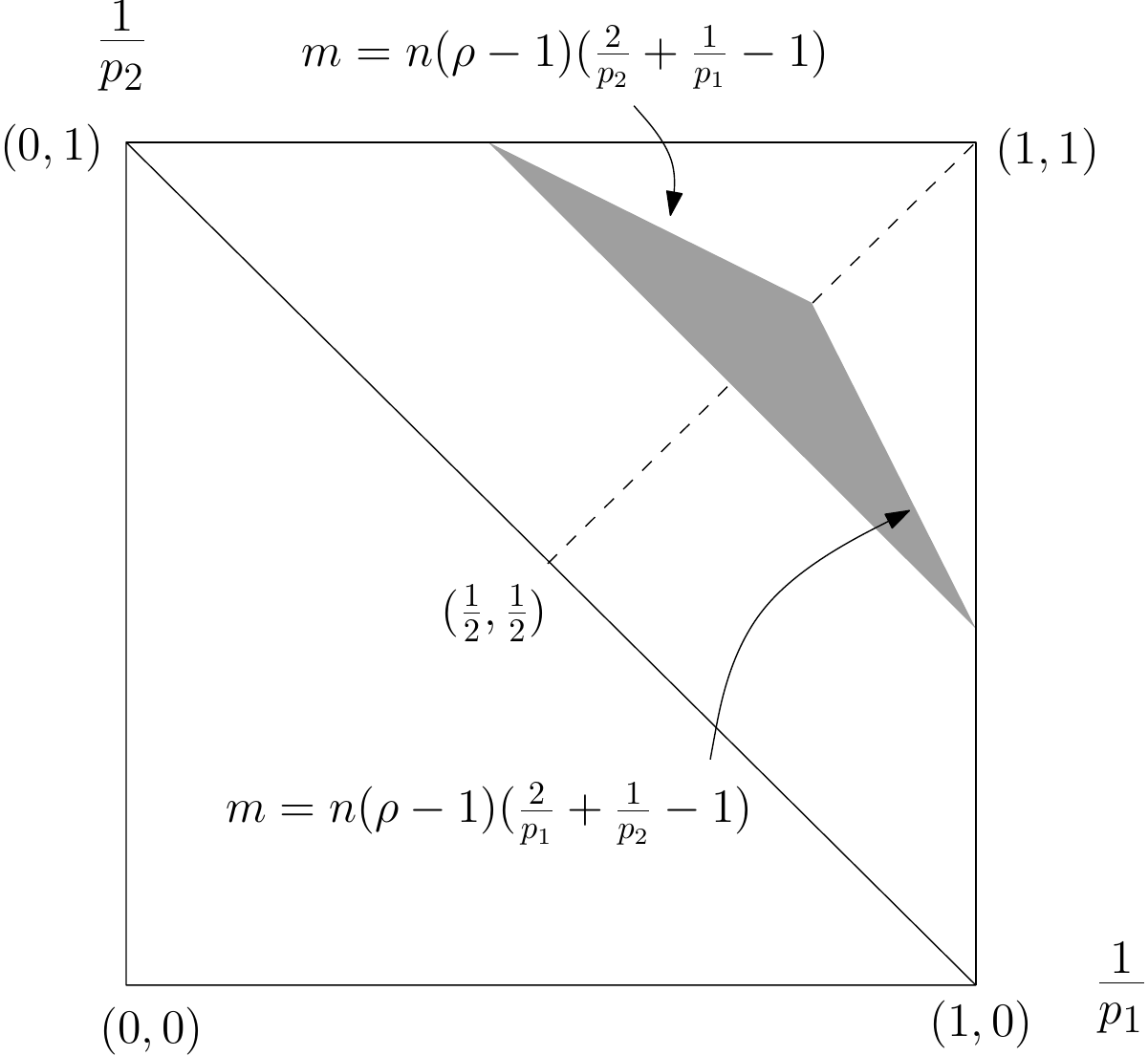} $\qquad$
    \includegraphics[width=2.7in]{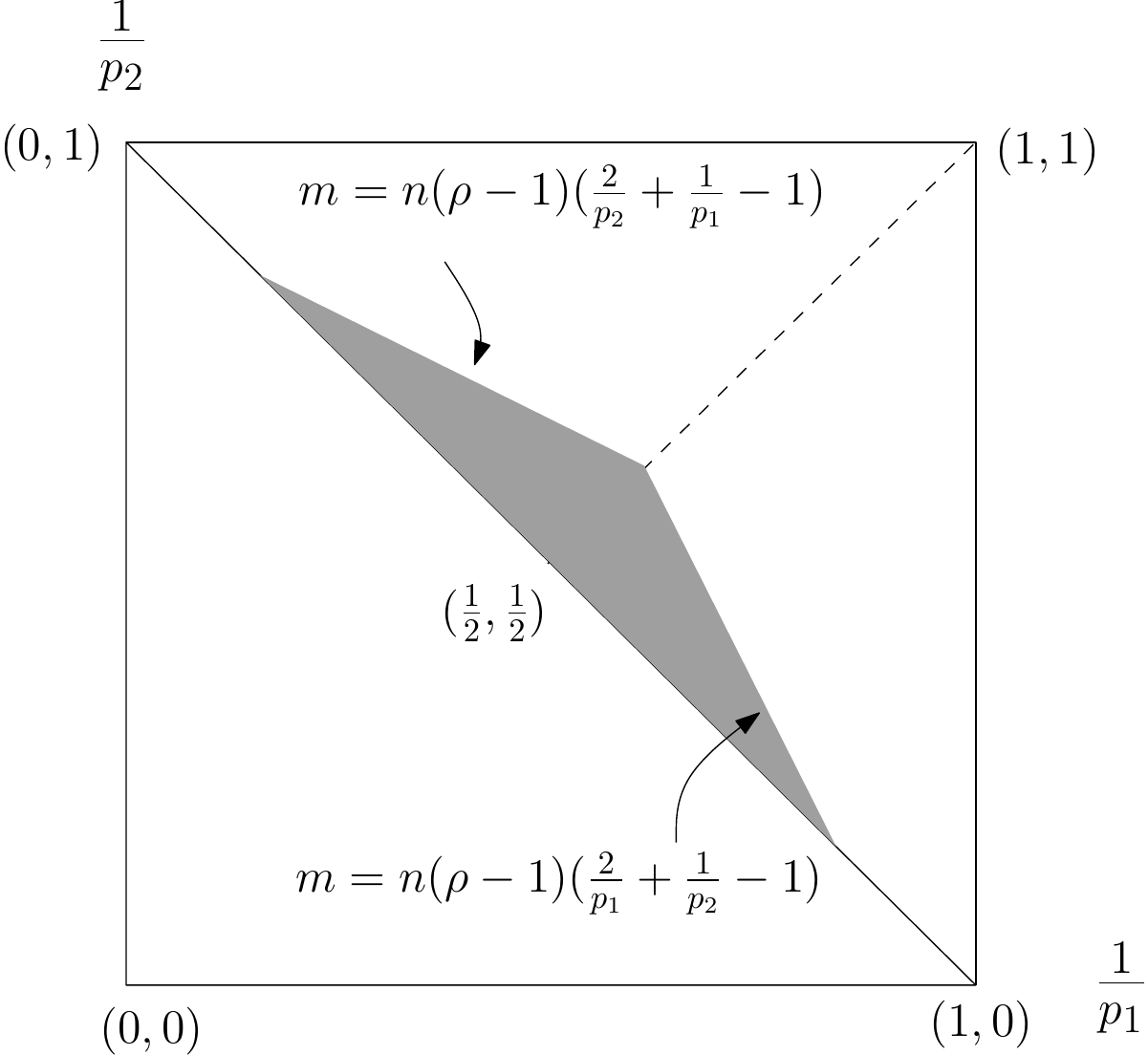}
   \end{center}
   \caption{Case $p<1$ of Theorem~\ref{thm:main2}}
   \label{fig2:main2}
\end{figure}

For Part (ii) consider the shaded triangle as  indicated in each case presented in
Figure~\ref{fig2:main2}. The value
$m(p_1,p_2)$ is constant, say $m,$ on the upper border of this
triangle, which is given by two segments with equations $m= n(\rho-1) (2/p_1+ 1/p_2 -
1)$ (inside triangle with vertices $(1,1),$
$(\frac{1}{2},\frac{1}{2}),$ and $(1,0)$) and $m= n(\rho-1) (2/p_2+
1/p_1 - 1)$ (inside triangle with vertices $(1,1),$
$(\frac{1}{2},\frac{1}{2}),$ and $(1,0)$). Then Part (ii) follows by bilinear real interpolation using
the weak type estimates obtained above for the vertices of the shaded triangle.

\end{proof}

\begin{remark}\label{caseLinfty} We note that the proof of Theorem~\ref{thm:main2} given for the case $p_1=p_2=p=\infty$ does not require any assumptions on the  derivatives of the symbol $\sigma$ in the space variables. This particular result is included in  \cite[Theorem 3.3]{MRS}, which  yields boundedness properties in Lebesgue spaces of bilinear pseudo-differential operators with rough symbols in the space variables as a consequence of  $\abs{T_\sigma(f,g)}$ being pointwise bounded
in terms of the Hardy-Littlewood  maximal operator evaluated at $f$ and $g.$
For completeness, we have  provided another proof of  the case $p_1=p_2=\infty$ of Theorem~\ref{thm:main2} following the arguments of the corresponding linear result in \cite{KeSt}.
\end{remark}

\section{Proof of Theorem \ref{thm:main3}} \label{sec:main3}

In this section we continue to use   $L^p$ and $\|\cdot\|_{L^p}$ to denote  the Lebesgue space $L^p(\re^n)$ and its norm, respectively. Sometimes it will be necessary to make explicit the variable of integration, say integration with respect to $x,$ in which case we employ  the notation $\|\cdot\|_{L^p(dx)}.$

\begin{proof}[Proof of Theorem \ref{thm:main3}]
Without loss of generality we may assume that the symbol $\sigma$
has compact support in the frequency variables $\xi$ and $\eta.$
Otherwise, define
$\sigma_\varepsilon(x,\xi,\eta):=\varphi(\varepsilon\,\xi,\varepsilon\,\eta)\sigma(x,\xi,\eta),$
where $\varphi$ is a smooth function compactly supported in $B(0,1)$
such that $0\le \varphi\le 1$ and $\varphi(0,0)=1.$ It easily
follows that $C(\sigma_\varepsilon)\lesssim C(\sigma)$ and  that
$\lim_{\varepsilon\to 0}T_{\sigma_\varepsilon}(f,g)=T_\sigma(f,g)$
pointwise for $f$ and $g$ belonging to the class $\mathcal{U}$ of
functions whose Fourier transforms are in $\mathcal{C}^\infty_0.$
Assuming the result for symbols of compact support, by Fatou's
lemma,
\[
\|T_\sigma(f,g)\|_{L^2}\le \liminf_{\varepsilon\to 0}\|T_{\sigma_\varepsilon}(f,g)\|_{L^2}\le C(\sigma)\,\|f\|_{L^2}\|g\|_{L^2},
\]
for $f,\,g\in\mathcal{U}.$ Since $\mathcal{U}$ is dense in $L^2$ the
desired result holds for non-compactly supported symbols as well.

Suppose  first that  $\sigma$ is $x$-independent and define
$\tau(\xi,\eta):=\sigma(x,\xi,\eta).$ Then
$$ \widehat{T_\tau(f,g)}(\xi) = \int_{\rn} \widehat{f}(\xi-\eta) \widehat{g}(\eta) \tau(\xi-\eta,\eta) d\eta,$$
and the Cauchy-Schwarz inequality yields
  \begin{equation*} |\widehat{T_\tau(f,g)}(\xi)| \lesssim \left(\int_{\re^n}|\hat{f}(\xi-\eta)|^2|\hat{g}(\eta)|^2\,d\eta\right)^{\frac{1}{2}} \sup_{\xi\in\rn} \|\tau(\xi-\cdot,\cdot)
  \|_{L^2}.
  \end{equation*}
Integrating in $\xi$ and using Plancherel's theorem it follows that
  \begin{equation} \|T_\tau(f,g)\|_{L^2} \lesssim \|f\|_{L^2} \|g\|_{L^2} \sup_{\xi\in\rn} \|\tau(\xi-\cdot,\cdot) \|_{L^2}, \label{eq:aze}
  \end{equation}
 which implies the desired result.

Next, we continue working with an $x$-independent symbol
$\tau(\xi,\eta)$ in order to get estimates that will be useful later
for
  $x$-dependent symbols.  Let   $\Phi$ be a smooth function compactly supported in $B(0,\sqrt{n})$ such that $0\le \Phi\le1$ and
  $$
  \sum_{k\in {\mathbb Z}^n} \Phi(k-x)=1,\quad  x\in\rn.
$$
For a function $h$ defined in $\re^n$ and $l\in\ent^n,$ we set
$h_l(x):=\Phi(x-l)h(x).$
  We will show that for every $N\in\na$
 \begin{align}
 \|\Phi(\cdot-l)T_\tau(f,g)\|_{L^2} \lesssim
 &  \sup_{\genfrac{}{}{0pt}{}{\xi\in\rn}{\abs{\alpha}\le 2N}} \|\partial_\xi^\alpha \tau(\xi-\cdot,\cdot) \|_{L^2}\sum_{j,k\in{\mathbb Z}^n} \frac{\|f_j\|_{L^2} \|g_k\|_{L^2}}{\left(1+|l-j|+|l-k|\right)^N}, \label{eq:offdiag}
\end{align}
for all $l\in\ent^n$ and with constants independent of $\tau,$ $l,$
$f$ and $g.$

We  have
 $$ \Phi(x-l)T_\tau(f,g)(x) = \sum_{j,k \in {\mathbb Z}^n} \Phi(x-l)T_\tau(f_j,g_k)(x),\quad x\in \re^n,$$
 and therefore \eqref{eq:offdiag} will follow from the estimate
 \begin{equation}\label{enough}
 \left\|\Phi(\cdot-l)T_\tau(f_j,g_k)\right\|_{L^2} \lesssim   \sup_{\genfrac{}{}{0pt}{}{\xi\in\rn}{\abs{\alpha}\le 2N}} \|\partial_\xi^\alpha \tau(\xi-\cdot,\cdot) \|_{L^2} \frac{\|f_j\|_{L^2} \|g_k\|_{L^2}}{\left(1+|l-j|+|l-k|\right)^N}.
\end{equation}

 Fix $l\in\ent^n.$ When $j,\,k\in\ent^n$ are such that $|l-j|+|l-k|\leq 10$, we  apply \eqref{eq:aze}:
\begin{align*}
 \left\|\Phi(\cdot-l)T_\tau(f_j,g_k)\right\|_{L^2} &  \leq \left\|T_\tau(f_j,g_k)\right\|_{L^2} \\
 & \lesssim \sup_{\xi\in\rn} \|\tau(\xi-\cdot,\cdot) \|_{L^2} \| f_j\|_{L^2} \|\ g_k\|_{L^2}  \\
 &  \lesssim  \sup_{\xi\in\rn} \|\tau(\xi-\cdot,\cdot)  \|_{L^2} \frac{\|f_j\|_{L^2} \|g_k\|_{L^2}}{\left(1+|l-j|+|l-k|\right)^N},
\end{align*}
for every integer $N$ and therefore \eqref{enough} holds.

We now consider $j$ and $k$ such that  $|l-j|+|l-k|\geq 10$ and,
without loss of generality, we assume  that $|l-j|\geq |l-k|$. Then
\begin{align*}
 T_\tau(f_j,g_k)(x) & = \int_{{\mathbb R}^{2n}} \left(\int_{{\mathbb R}^{2n}} e^{i\left(\xi\cdot(x-y)+\eta\cdot(x-z)\right)} \tau(\xi,\eta) d\xi d\eta\right) f_j(y) g_k(z) dy dz \\
& = \int_{{\mathbb R}^{2n}} \left(\int_{{\mathbb R}^{2n}} e^{i\left(\xi\cdot(x-y)+\eta\cdot(x-z)\right)}
(1-\Delta_{\xi})^{N} \tau(\xi,\eta) d\xi d\eta\right) \frac{f_j(y) g_k(z) \,dy dz}{(1+|x-y|^2)^{N}}\\
&=\int_{\re^{2n}}\mathcal{F}_{2n}((1-\Delta_{\xi})^{N} \tau)(y-x,z-x)\frac{f_j(y) g_k(z) \,dy dz}{(1+|x-y|^2)^{N}},
\end{align*}
where $\mathcal{F}_{2n}$ denotes the Fourier transform in
$\re^{2n}.$
 Multiplying by $\Phi(x-l)$ and using the Sobolev embedding
$W^{P,2} \subset L^\infty$ for any $P>n/2$, by fixing $x\in \re^n,$
it follows that
\begin{align*}
& \left|\Phi(x-l) T_\tau(f_j,g_k)(x)\right| \\
 &  \lesssim \sup_{a\in\re^n} \left|\Phi(a-l) \int_{{\mathbb R}^{2n}} \mathcal{F}_{2n}((1-\Delta_{\xi})^{N} \tau)(y-x,z-x)  \frac{f_j(y)  g_k(z)}{(1+|a-y|^2)^{N}} dy dz\right|\\
 & \lesssim  \sup_{|\beta|\leq P} \left\|\int_{{\mathbb R}^{2n}} \mathcal{F}_{2n}((1-\Delta_{\xi})^{N} \tau)(y-x,z-x)\,
(\partial_a^\beta\gamma_{l,N})(a,y)\,f_j(y)  g_k(z) dy dz\right\|_{L^2(da)},
\end{align*}
where $\gamma_{l,N}(a,y):= \frac{\Phi(a-l)}{(1+|a-y|^2)^{N}}.$
Therefore,
\begin{align*}
& \left\|\Phi(\cdot-l)T_\tau(f_j,g_k)\right\|_{L^2}  \\
& \lesssim  \sup_{|\beta|\leq P} \left\|\int_{{\mathbb R}^{2n}} \mathcal{F}_{2n}((1-\Delta_{\xi})^{N} \tau)(y-x,z-x) \, (\partial_a^\beta\gamma_{l,N})(a,y)\,f_j(y)  g_k(z) dy dz\right\|_{L^2(dadx)}\\
&  = \sup_{|\beta|\leq P} \left\| \,T_{((1-\Delta_{\xi})^{N} \tau)}\left(\partial_a^\beta \left(\frac{\Phi(a-l) }{(1+|a-\cdot|^2)^{N}}\right)f_j(\cdot), g_k(\cdot)\right) (x)
  \right\|_{L^2(dadx)}.
\end{align*}
Applying (\ref{eq:aze}) to $T_{((1-\Delta_{\xi})^{N} \tau)}$ then
yields,
\begin{align*}
& \left\|\Phi(\cdot-l)T_\tau(f_j,g_k)\right\|_{L^2}  \\
& \lesssim \sup_{\xi\in\rn} \|(1-\Delta_{\xi})^{N} \tau(\xi-\cdot,\cdot)\|_{L^2}  \sup_{\abs{\beta}\le P}\left\|\partial_a^\beta \left(\frac{\Phi(a-l) }{(1+|a-y|^2)^{N}}\right)f_j(y)\right\|_{L^2(dady)} \|g_k\|_{L^2}\\
& \lesssim \sup_{\xi\in\rn} \|(1-\Delta_{\xi})^{N} \tau(\xi-\cdot,\cdot)\|_{L^2} \frac{\|f_j\|_{L^2} \|g_k\|_{L^2}}{(1+\abs{l-j}^2)^N},
\end{align*}
 giving \eqref{enough}, where we have used that
\[
\left\|\partial_a^\beta \left(\frac{\Phi(a-l) }{(1+|a-y|^2)^{N}}\right)\right\|_{L^2(da)}\lesssim \frac{1}{(1+|l-j|^2)^N}, \qquad y\in B(j,\sqrt{n}).
\]

 Consider now an $x$-dependent symbol. Then
$$T_\sigma(f,g)(x)= U_{x}(f,g)(x),$$
where
\begin{align*}
U_{y}(f,g)(x):= \int_{{\mathbb R}^{2n}} e^{ix\cdot(\xi+\eta)} \widehat{f}(\xi)\widehat{g}(\eta)\sigma(y,\xi,\eta)d\xi d\eta, \quad x,\, y\in \re^n.
\end{align*}
Fixing $x\in\re^n,$ $l\in\ent^n,$ and using the Sobolev embedding
$W^{s,2} \hookrightarrow L^\infty$ for an integer $s>n/2$, we get
\begin{align*}
 \left|\Phi(x-l) T_\sigma(f,g)(x) \right|&\le  \sup_{y\in\re^n}|\Phi(y-l) U_{y}(f,g)(x)|\\
 &\le \sum_{|\beta|\leq s} \|\partial^\beta_y\left(\Phi(y-l) U_{y}(f,g)(x)\right)\|_{L^2(dy)}\\
 &\lesssim \sum_{|\beta|\leq s} \|\chi_{B(l)}(y)\partial^\beta_yU_{y}(f,g)(x)\|_{L^2(dy)},
 \end{align*}
 where $B(l)=B(l,\sqrt{n}).$ Let $\tilde{\Phi}$ be a smooth function supported in $B(0,\sqrt{n})$ such that $\tilde{\Phi}\Phi=\Phi.$ Multiplying by $\tilde{\Phi}(x-l),$
integrating in $x$ and using Fubini's Theorem, we obtain
$$ \left\|\Phi(\cdot-l) T_\sigma(f,g)\right\|_{L^2} \lesssim \sum_{|\beta|\leq s}  \left\|\chi_{B(l)}(y) \left\|\chi_{B(l)}(x)\tilde{\Phi}(x-l)\partial^\beta_y U_{y}(f,g)(x) \right\|_{L^2(dx)} \right\|_{L^2(dy)}.$$
For each  $\beta\in\na^n_0,$ $\abs{\beta}\le s,$ and $y\in\re^n,$ we
look at $\partial^\beta_y U_{y}$ as  the bilinear multiplier defined
by the $x$-independent symbol
$$\tau_y^\beta(\xi,\eta):= \partial_{y}^{\beta} \sigma(y,\xi,\eta).$$
Then applying (\ref{eq:offdiag}), which also holds if  on its left
hand side $\Phi$ is replaced by $\tilde{\Phi}$, we deduce
\begin{align*}
& \left\| \Phi(\cdot-l)T_\sigma(f,g)\right\|_{L^2} \\
 &\lesssim\sum_{|\beta|\leq s} \sum_{j,k\in {\mathbb Z}^n}   \sup_{\genfrac{}{}{0pt}{}{\xi\in\rn}{|\alpha|\le 2N}}
 \sup_{y\in \re^n} \|\partial_\xi^\alpha \tau_y^\beta(\xi-\cdot,\cdot) \|_{L^2}
   \frac{\|f_j\|_{L^2} \|g_k\|_{L^2}}{\left(1+|l-j|+|l-k|\right)^N},
\end{align*}
which implies
$$ \left\|\Phi(\cdot-l) T_\sigma(f,g)\right\|_{L^2} \lesssim C(\sigma)\sum_{j,k\in\ent^n} \frac{\|f_j\|_{L^2} \|g_k\|_{L^2}}{\left(1+|l-j|+|l-k|\right)^N} ,$$
with
$$ C(\sigma):= \sup_{\genfrac{}{}{0pt}{}{|\beta|\le s}{|\alpha|\le 2N}} \sup_{\xi,y\in\rn}  \|\partial_\xi^\alpha \partial_{y}^{\beta} \sigma(y,\xi-\cdot,\cdot) \|_{L^2}.$$

Using H\"older's inequality we then obtain that
\begin{align*}
& \left\|\Phi(\cdot-l) T_\sigma(f,g)\right\|^2_{L^2}\\
& \lesssim C(\sigma)^2\sum_{j,k\in\ent^n} \frac{\|f_j\|^2_{L^2} \|g_k\|^2_{L^2}}{\left(1+|l-j|+|l-k|\right)^N}  \sum_{j,k\in\ent^n} \frac{1}{\left(1+|l-j|+|l-k|\right)^N}.
\end{align*}
Choosing  $N>2n,$ the second sum on the right hand side is finite
and   after summing over $l\in{\mathbb Z}^n$, we conclude that
$$\sum_{l\in\ent^n}\left\|T_\sigma(f,g)_l\right\|^2_{L^2}=\sum_{l\in\ent^n} \left\|\Phi(\cdot-l) T_\sigma(f,g)\right\|^2_{L^2} \lesssim C(\sigma)^2\sum_{j\in\ent^n} \|f_j\|^2_{L^2} \sum_{k\in\ent^n} \|g_k\|^2_{L^2}.$$
The desired result follows by taking $N=2n+1,$  $s=[\frac{n}{2}]+1,$
and noting that $\|h\|^2_{L^2}\sim \sum_{j} \|h_j\|^2$.
\end{proof}

\section{Proof of Theorem \ref{thm:main4}} \label{sec:main4}

The following lemmas, whose proof are included in  Section~\ref{sec:prooflemmas}, will be used to prove Theorem \ref{thm:main4}.

\begin{lemma}\label{compactsupp} Let $m\in\re,$ $0\le \delta,\,\rho\le 1,$   $\sigma \in \bs$ and $N>n.$
\begin{enumerate}[(a)]
\item If $0<R\le 1$ and  $\supp (\sigma) \subset
\{\xxe: \absx +\abse \leq R \}$ then
\[
\norm{T_\sigma(f,g)}{L^\infty} \lesssim R^{2n} \norm{\sigma}{0,2N} \norm{f}{L^\infty}
\norm{g}{L^\infty},\qquad f,\,g\in{L^\infty}.
\]
\item If  $ R\ge 1$  and $\supp(\sigma) \subset \{ R \leq \absx + \abse \leq 4 R\}$ then
\[
\norm{T_\sigma(f,g)}{L^\infty} \lesssim R^{(1-\rho)n + m} \norm{\sigma}{0,2N}\norm{f}{L^\infty}\norm{g}{L^\infty}, \qquad f,\,g\in{L^\infty}.
\]
\end{enumerate}
\end{lemma}

\begin{lemma}\label{smallsupp} Let $Q \subset \re^n$ be a cube with diameter $d$ and
$\sigma \in \bs$ with $m = n (\rho-1),$ $0\le \delta,\,\rho\le 1,$ such that
\[
\supp (\sigma) \subset \{ \xxe : \absx + \abse \leq  d^{-1} \}.
\]
Then, for every $N > n$,
\[
\frac{1}{\abs{Q}} \int \abs {T_\sigma(f,g)(x) - T_\sigma(f,g)_Q} \, dx \lesssim
\norm{\sigma}{1,2N} \norm{f}{L^\infty} \norm{g}{L^\infty},\qquad f,\,g\in\mathcal{S},
\]
with constants only depending on $n,\,N,\,\rho$ and $\delta$. Here $T_\sigma(f,g)_Q$ is the average of
$T_\sigma(f,g)$ over $Q$.
\end{lemma}

\begin{lemma}\label{bigsupport1} Let  $d>0$ and
 $\sigma \in \bs,$ $m = n(\rho-1),$ $0\le \delta,\rho\le 1,$ such that
\[
\supp (\sigma) \subset \{ \xxe : \absx + \abse \geq d^{-1} \}.
\]
Let $\phi \in \mathcal{S},$  $\phi \geq 0$,  and $\supp{( \hat
\phi)} \subset \{z \in \re^n : \abs{z} \leq \frac{1}{8}d^{-\rho}\}.$ For $f,g \in \mathcal{S}$,
define
\[
R(f,g)(x):=\phi^2(x)T_\sigma(f,g)(x)-T_\sigma(\phi f,\phi g)(x).
\]
Then, for every $N > n$, we have
\[
\norm{R(f,g)}{L^\infty} \lesssim \norm{\sigma}{0,2N+1} \norm{f}{L^\infty}\norm{g}{L^\infty}, \qquad f,\,g\in\mathcal{S},
\]
with constants only depending on $n,\,N,\,\rho$ and $\delta$.
\end{lemma}

\begin{remark} The proofs of the above lemmas show that $BS^{m}_{\rho,\delta}$ can be replaced by $BS^{m}_{\rho,\delta,0,2N},$ $BS^{m}_{\rho,\delta,1,2N}$ and $BS^{m}_{\rho,\delta,0,2N+1},$ respectively (see definition of these spaces in Section~\ref{lessderivatives}).
\end{remark}

\begin{proof}[Proof of Theorem \ref{thm:main4}]
 Given $\sigma \in BS^{m}_{\rho, 0}$, with $m=n(\rho-1)$, we have to
 prove that
 \begin{equation}\label{LinfBMO}
\frac{1}{\abs{Q}} \int \abs {T_\sigma(f,g)(x) - T_\sigma(f,g)_Q} \, dx \lesssim  \norm{f}{L^\infty} \norm{g}{L^\infty},
\end{equation}
for all cubes $Q\subset \re^n$  and $f,\,g\in \mathcal{S}.$

Let $Q$ be a cube with diameter $d$ and assume first  that $d \leq
1.$ We write
$$
\sigma(x,\xi,\eta) = \sigma(x,\xi,\eta) (1-\theta(\xi,\eta)) + \sigma(x,\xi,\eta) \theta(\xi,\eta) =: \sigma_1 + \sigma_2,
$$
where $\theta : \rn \times \rn \rightarrow \re$ is a smooth, non-negative function, $\theta(\xi,\eta)=\tilde{\theta}(d\,\xi,d\,\eta)$ with
$$
\supp(\tilde{\theta}) \subset \{(\xi, \eta)\in \rn \times \rn : |\xi|+|\eta| \geq 1\},
$$
and $\tilde{\theta} \equiv 1$ in $\{(\xi, \eta)\in \rn \times \rn :
|\xi|+|\eta| \geq 2 \}$. Since $d \leq 1,$ then $\sigma_1$, $\sigma_2 \in
BS^{m}_{\rho, 0}$ satisfy
\begin{equation}
\norm{\sigma_j}{K,M} \lesssim \norm{\sigma}{K,N}, \quad K,M \in \na_0, j=1,2,
\end{equation}
with constants independent of $d$ and $\sigma.$

Let $\phi$ be as in Lemma \ref{bigsupport1}, this is, $\phi \in
\mathcal{S},$  $\phi \geq 0$,  and $\supp{(
\hat \phi)} \subset \{z \in \re^n : \abs{z} \leq d^{-\rho}/8\}.$ In
addition,  we  assume $\phi \equiv 1$ on $Q$ and, in accordance with the uncertainty principle, we choose $\phi$ such that
$
\|\phi\|_{L^2}\lesssim d^{\frac{n\rho}{2}}.
$
 For $x \in Q$ we have
\begin{align*}
T_\sigma(f,g)(x)& = T_{\sigma_1}(f,g)(x) + T_{\sigma_2}(f,g)(x) =   T_{\sigma_1}(f,g)(x) + \phi^2(x)T_{\sigma_2}(f,g)(x)\\
& = T_{\sigma_1}(f,g)(x) + T_{\sigma_2}(\phi f, \phi g)(x) + R(f,g)(x),
\end{align*}
where $R(f,g)(x) =  \phi^2(x)T_{\sigma_2}(f,g)(x) -
T_{\sigma_2}(\phi f, \phi g)(x).$

In order to get \eqref{LinfBMO}, it is  enough to prove the inequality
\begin{equation}\label{reduc1}
\norm{T_{\sigma_2}(\phi f, \phi g)}{L^1(Q)} \lesssim \, \|\sigma_2\|_{K,M}\,d^n \norm{f}{L^\infty} \norm{g}{L^\infty}, \quad f, g \in \mathcal{S},
\end{equation}
for some $K,\,M\in\na_0.$
Indeed, using \eqref{reduc1} and Lemmas \ref{smallsupp} and \ref{bigsupport1}, for $N > n$,  we write
\begin{align*}
&\frac{1}{|Q|} \int_Q |T_\sigma(f,g)(x) - T_\sigma(f,g)_Q|\, dx\\
&  \leq \frac{1}{|Q|} \int_Q |T_{\sigma_1}(f,g)(x) - T_{\sigma_1}(f,g)_Q|\, dx
+ \frac{2}{|Q|} \norm{T_{\sigma_2}(\phi f, \phi g)}{L^1(Q)} + 2 \norm{R(f,g)}{L^\infty}\\
& \lesssim \left( \norm{\sigma_1}{1,2N} +  \|\sigma_2\|_{K,M} + \norm{\sigma_2}{0,2N+1}\right) \norm{f}{L^\infty} \norm{g}{L^\infty},
\end{align*}
 and therefore \eqref{LinfBMO} holds when the diameter of  $Q$ is less than or equal to 1.

In turn, (\ref{reduc1}) will follow from
 \begin{equation}\label{reduc1-bis}
 \norm{T_{\sigma_2}(\phi f, \phi g)}{L^2} \lesssim  \|\sigma_2\|_{K,M}\, d^{\frac{n}{2}} \norm{f}{L^\infty} \norm{g}{L^\infty}, \quad f, g \in \mathcal{S},
 \end{equation}
 since
 $$ \frac{1}{|Q|} \norm{T_{\sigma_2}(\phi f, \phi g)}{L^1(Q)} \leq \frac{1}{|Q|^{1/2}} \norm{T_{\sigma_2}(\phi f, \phi g)}{L^2(Q)} \leq \frac{1}{|Q|^{1/2}} \norm{T_{\sigma_2}(\phi f, \phi g)}{L^2}.$$
Moreover, because $\phi$ satisfies $\|\phi\|_{L^2} \lesssim
d^{\frac{\rho n}{2}}$, (\ref{reduc1-bis}) can be reduced to proving
that
 \begin{equation}\label{reduc1-2}
 \|T_{\sigma_2}\|_{L^2 \times L^2 \rightarrow L^2} \lesssim  \|\sigma_2\|_{K,M}\,d^{\frac{n}{2}-\rho n}.
 \end{equation}
By Theorem \ref{thm:main3}, the support of $\sigma_2,$ and the fact
that $\sigma_2\in BS^{m}_{\rho,0}$ with $m=n(\rho-1)$ and
$0<\rho<\frac{1}{2},$ we obtain
 \begin{align*}
  \|T_{\sigma_2}\|_{L^2 \times L^2 \rightarrow L^2} & \lesssim \sup_{\genfrac{}{}{0pt}{}{|\beta|\le [\frac{n}{2}]+1}{|\alpha|\le 2(2n+1)}} \sup_{y,\xi\in\rn}  \|\partial_\xi^\alpha \partial_{y}^{\beta} \sigma_2(y,\xi-\cdot,\cdot) \|_{L^2} \\
  & \lesssim \sup_{\xi\in\rn} \|\chi_{\{|\xi-\eta|+|\eta|\geq d^{-1}\}}(\xi,\eta) \left(1+|\xi-\eta|+|\eta|\right)^{m}\|_{L^2(d\eta)} \\
  & \lesssim  \|\sigma_2\|_{K,M}\,\left(\int_{|\eta|\geq d^{-1}} |\eta|^{2m} d\eta\right)^{1/2} + \left(\int_{|\eta|\leq d^{-1}} d^{-2m}  d\eta\right)^{1/2} \\
  & \lesssim  \|\sigma_2\|_{K,M}\,d^{-m-\frac{n}{2}}=  \|\sigma_2\|_{K,M}\,d^{\frac{n}{2}-\rho n},
 \end{align*}
where we have taken $K= [\frac{n}{2}]+1$ and $M= 2(2n+1).$

 The case $d > 1$ follows using the decomposition of $\sigma$ with $\theta=\tilde{\theta}$ and then proceeding analogously but
applying to the term corresponding to $T_{\sigma_1}$  Lemma~\ref{compactsupp} instead of Lemma~\ref{smallsupp}.
\end{proof}

\section{Proof of Theorem \ref{thm:bssobolev}}\label{sec:bssobolev}

For $s>0$, we recall the bilinear fractional integral operator of order $s>0,$ introduced in Kenig-Stein~\cite{KS}, defined by
\begin{align}\label{def:multfracop}
\mathcal{I}_{s} (f,g)(x) &:= \int_{\re^{2n}}
 \frac{f(y) g(z)}{(|x-y|+|x-z|)^{2n-s}}\,dydz,\quad x\in\re^n.
\end{align}
It easily follows that
\[
\mathcal{I}_s(f,g)(x)\le I_{s_1}(f)(x)\,I_{s_2}(g)(x) ,\qquad x\in\re^n,\,s_1+s_2=s,
\]
where
\[
I_\tau(h)(x)=\int_{\re^n}\frac{h(y)}{\abs{x-y}^{n-\tau}}\,dy,\quad 0<\tau<n,
\]
is the linear fractional integral. The boundedness properties of $I_{\tau},$ $0<\tau<n,$ and H\"older's inequality imply that $\mathcal{I}_s$ is bounded form $L^{p_1}\times L^{p_2}$ into $L^p$ with $\frac{1}{p}=\frac{1}{p_1}+\frac{1}{p_2}-\frac{s}{n},$ $0<s<2n,$ $1<p_1,\,p_2<\infty,$ $q>0.$

We now observe that if $\sigma\in\bs,$ $m\le 2n(\rho-1)-\rho s,$ $0<s<2n,$ then part \eqref{thm:kernelestimates-v} of  Theorem~\ref{thm:kernelestimates} implies that
\begin{equation}\label{eqn:boundbyfracint}
\abs{T_\sigma(f,g)(x)}\lesssim \abs{\mathcal{I}_s(f,g)(x)}.
\end{equation}
Therefore Theorem~\ref{thm:bssobolev} follows from this inequality and the boundedness properties of $\mathcal{I}_s.$ The case $\rho =1$ of
Theorem~\ref{thm:bssobolev} was treated in \cite{bmmn}.

\section{Proof of lemmas from Section~\ref{sec:main4}}\label{sec:prooflemmas}

\begin{proof}[Proof of Lemma~\ref{compactsupp}]
We have
\begin{equation}\label{intrepr}
T_\sigma(f,g)(x)=  \int_{\re^{2n}} \mathcal{K}(x,x-y,x-z) f(y) g(z) \,dy \, dz,
\end{equation}
where
\[
\mathcal{K}(x,y,z)= \int_{\re^{2n}} \ei{\xi}{y} \ei{\eta}{z} \sym \dxde=\mathcal{F}^{-1}_{2n}(\sigma(x,\cdot,\cdot))(y,z),
\]
and  $\mathcal{F}_{2n}$ denotes the inverse Fourier transform in $\re^{2n}.$
 Then, it is enough to show  that for  $N>n,$ $N\in\na_0,$
\begin{equation}\label{partA}
\sup_{x\in\re^n} \int_{\re^{2n}} \abs{\mathcal{K}(x,y,z)} \, dy \, dz \lesssim \, R^{2n}\,\norm{\sigma}{0,2N}.
\end{equation}
 and
\begin{equation}\label{partB}
\sup_{x\in\re^n}\int_{\re^{2n}} \abs{\mathcal{K}(x,y,z)} \, dy \, dz \lesssim  R^{(1-\rho) n + m} \norm{\sigma}{0,2N}.
\end{equation}
for part (a) and part (b), respectively. (Note that this allows to extend $T_\sigma$ to a bounded operator  form $L^\infty\times L^\infty$ into $L^\infty$ by using the representation \eqref{intrepr} to define $T_\sigma(f,g)$ for $f,g\in L^\infty$).

Since  $\sigma$ is a smooth function with  compact support in $\xi$ and $\eta$ we have
\begin{align}\label{intparts}
(1+\abs{(y,z)}^{2})^{N}\mathcal{K}(x,y,z)&=\int_{\re^{2n}}\sigma(x,\xi,\eta) \,(1-\Delta_\xi-\Delta_\eta)^{N}(e^{i\xi \cdot y}\, e^{i\eta\cdot  z})\,d\xi d\eta\\
&=\int_{\re^{2n}}(1-\Delta_\xi-\Delta_\eta)^{N}(\sigma(x,\xi,\eta)) \,e^{i\xi \cdot y}\, e^{i\eta \cdot z}\,d\xi d\eta\nonumber\\
&=\mathcal{F}^{-1}_{2n}((1-\Delta_\xi-\Delta_\eta)^{N}(\sigma(x,\cdot,\cdot)) )(y,z),\nonumber
\end{align}
and similarly,
\begin{align}\label{intparts2}
\abs{(y,z)}^{2N}\mathcal{K}(x,y,z)&=\mathcal{F}^{-1}_{2n}((-\Delta_\xi-\Delta_\eta)^{N}(\sigma(x,\cdot,\cdot)) )(y,z).
\end{align}

For part (a), we use \eqref{intparts} and that $R\le 1$ to get,
\[
\abs{\mathcal{K}(x,y,z)} \lesssim \frac{R^{2n}\norm{\sigma}{0,2N}}{(1+\abs{(y,z)}^2)^{N} }
\]
and then \eqref{partA} follows since $N > n.$

For part (b) we write
\begin{align*}
\int_{\re^{2n}} \abs{\mathcal{K}(x,y,z)}\, dy dz  = \mathop{\int}_{|y|+|z|\leq R^{-\rho}} \abs{\mathcal{K}(x,y,z)} \, dy dz +
\mathop{\int}_{|y|+|z|\geq R^{-\rho}} \abs{\mathcal{K}(x,y,z)}\,  dy dz.
\end{align*}
Let us now estimate the first integral. By Cauchy-Schwarz inequality,  Plancherel's identity and the fact that $R\ge 1,$ we have
\begin{align*}
 \left(\mathop{\int }_{|y|+|z|\leq R^{-\rho}} \abs{\mathcal{K}(x,y,z)} \,dy dz  \right)^2 & \lesssim\, R^{- 2 \rho n}
\mathop{\int }_{|y|+|z|\leq R^{-\rho}} \abs{\mathcal{K}(x,y,z)}^2\, dy dz\\
 &  \lesssim  R^{-2\rho n} \mathop{\int }_{|\xi|+|\eta|\sim R} \left| \sigma(x,\xi,\eta) \right|^2 d\xi d\eta\\
 & \lesssim \norm{\sigma}{0,0}^2 R^{-2\rho n} \mathop{\int }_{|\xi|+|\eta|\sim R} (1+ |\xi|+|\eta|)^{2m} d\xi d\eta\\
& \lesssim \norm{\sigma}{0,0}^2 R^{-2\rho n} R^{2m+ 2n} =\norm{\sigma}{0,0}^2 R^{2((1-\rho)n + m)}.
\end{align*}
Next, we estimate the second integral. Multiplying and dividing by $\abs{(y,z)}^{2N},$ and using the Cauchy-Schwarz inequality,  that $N>n,$ \eqref{intparts2},   Plancherel's identity, and that $R\ge 1,$ it follows that
\begin{align*}
 \left(\mathop{\int }_{|y|+|z|\geq R^{-\rho}} \abs{\mathcal{K}(x,y,z)} \, dy dz  \right)^2
\lesssim & \left(\mathop{\int}_{|y|+|z|\geq R^{-\rho}} \frac{1}{|(y,z)|^{4N}} dy dz\right)\\&\times \left( \mathop{\int }_{|y|+|z|\geq R^{-\rho}} \abs{ |(y,z)|^{2N}\mathcal{K}(x,y,z)}^2
dy dz\right)\\
 \lesssim & \,R^{\rho(4N-2n)} \mathop{\int}_{|\xi|+|\eta|\sim R}\abs{(-\Delta_\xi-\Delta_\eta)^{N}\sigma(x,\xi,\eta)) }^2\,d\xi d\eta\\
 \lesssim &\,\norm{\sigma}{0,2N}^2 R^{\rho(4N-2n)} \!\!\!\!\!\mathop{\int}_{|\xi|+|\eta|\sim R} (1 + |\xi| + |\eta|)^{2(m - \rho 2 N)} d\xi d\eta \\
 \lesssim& \,\norm{\sigma}{0,2N}^2 R^{\rho(4N-2n)} R^{2(m - \rho 2N+n)} \\
 = & \, \norm{\sigma}{0,2N}^2  R^{2((1-\rho)n + m)}.
\end{align*}

The last two computations give \eqref{partB}.
\end{proof}

\begin{proof}[Proof of Lemma~\ref{smallsupp}]  Let $Q,\,d,\,N, \,m$ and $\sigma$ be as in the hypothesis. By definition,
\[
T_\sigma(f,g)(x) = \int_{\re^{2n}} \sym \hat f(\xi) \hat g(\eta)\,\ei{x}{(\xi + \eta)} \,d\xi
\,d\eta,\quad f,\,g\in\mathcal{S}.
\]
Hence, for a fixed $j=1,\ldots,n$, the bilinear symbol $\tau=\tau(x,\xi,\eta)$ of the bilinear operator $
\frac{\partial T_\sigma(f,g)}{\partial x_j}$ is given by
\[
\tau \xxe = i (\xi_j + \eta_j) \sym +
 \frac{\partial \sigma}{\partial x_j} \xxe.
\]
Then symbol $\tau$ is also supported in $\{ \xxe : \absx + \abse \leq d^{-1} \}$ and $\tau\in BS^{m+\delta}_{\rho,\delta}.$ Elementary computations show that for $K,\,M\in\na_0,$
\begin{equation}\label{elementary}
\norm{\tau}{K,M}\le\max(1,d^{-1}) \norm{\sigma}{K+1,M},
\end{equation}
where $\norm{\tau}{K,M}$ corresponds to a norm of $\tau$ as an element of $BS^{m+\delta}_{\rho,\delta},$ while $\norm{\sigma}{K+1,M}$ corresponds to a norm of $\sigma$ as an element of $BS^{m}_{\rho,\delta}.$ Then
\begin{align*}
\int_Q |T_\sigma(f,g)(x)-T_\sigma(f,g)_Q|\dx
&= \frac{1}{\abs{Q}}  \int_Q \left| \int_Q
\left( T_\sigma(f,g)(x)-T_\sigma(f,g)(y) \right)\dy \right| \dx\\
& \leq  d\, \abs{Q}\norm{\nabla T_\sigma(f,g)}{L^\infty}  \lesssim d\, \abs{Q}\norm{T_\tau(f,g)}{L^\infty} \\
&\lesssim  d\, \abs{Q} \,\min(1,d^{-2n})\, \norm{\tau}{0,2N}\norm{f}{L^\infty}
\norm{g}{L^\infty} \\
&\lesssim   \abs{Q} \, \norm{\sigma}{1,2N}\norm{f}{L^\infty}
\norm{g}{L^\infty}\quad (\tm{by \eqref{elementary}}),
\end{align*}
where we have used Lemma \ref{compactsupp}.  The result follows.

\end{proof}

\begin{proof}[Proof of Lemma~\ref{bigsupport1}]  Let $d,\,N, \,m,$ $\phi$ and $\sigma$ be as in the hypothesis.  We notice that the bilinear symbol $\theta(x,\xi,\eta)$ of $R$ is given by
\begin{align*}
& \theta(x,\xi,\eta)=
 \int_{\re^{2n}} \ei{x}{(y+z)}\left(\sym - \sigma(x,\xi+y,\eta+z)\right)  \hat \phi(y) \hat
\phi(z) \dy \dz.
\end{align*}

 We first  assume that $d \leq 1$ and note that $\supp(\theta)\subset\{(x,\xi,\eta):|\xi|+|\eta|\ge \frac{1}{2}d^{-1}\}.$
Consider a partition of unity of $\re^{2n}$ given by $\{\psi_k\}_{k\in\na_0},$
\[\sum_{k\ge 0}\psi_k(\xi,\eta)=1,\quad \xi,\eta\in\re^{n},\]
 where  $\psi_0 \in \s (\re^{2n})$ is supported in the set $\{(\xi,\eta): \xpe \leq 2 d^{-1}\}$ and
$\psi_k(\xi,\eta)= \psi(d 2^{-k} \xi, d 2^{-k} \eta)$ with $\psi\in\mathcal{S}(\re^{2n})$ and
  $\supp (\psi) \subset  \{(\xi,\eta):1/2 \leq \xpe \leq 2\}$ for $k\ge 1.$ Then $\supp (\psi_k)\subset\{(\xi,\eta):\xpe \sim 2^k d^{-1}\}$ for $k\ge 1$ and
\[
\theta \xxe = \sum_{k \geq 0}  \theta_k(x,\xi,\eta),
\]
where $\theta_k\xxe := \theta \xxe \psi_k(\xi,\eta)$. We will show that for all integers
$M, k \geq 0$
\begin{equation}\label{thetak}
\norm{\theta_k}{0,M} \lesssim 2^{-\rho k}\norm{\sigma}{0,M+1},
\end{equation}
with constants depending only on  $M,\,n,\,\rho,$ and $\delta.$

Define $R_k$ as the bilinear pseudo-differential operator with kernel $\theta_k$. The
lemma will follow from (\ref{thetak}). Indeed,
\begin{align*}
 \norm{R(f,g)}{L^\infty} & \leq \sum_{k \geq 0} \norm{R_k(f,g)}{L^\infty} \lesssim \sum_{k \geq 0}  \norm{\theta_k}{0,2N} \norm{f}{L^\infty} \norm{g}{L^\infty}
& (\tm{by Lemma \ref{compactsupp}})  \\
& \lesssim \sum_{k \geq 0} 2^{-\rho k}\,\norm{\sigma}{0,2N+1}  \norm{f}{L^\infty} \norm{g}{L^\infty}
& (\tm{by \eqref{thetak}})\\
& \lesssim \norm{\sigma}{0,2N+1} \norm{f}{L^\infty} \norm{g}{L^\infty}.
\end{align*}

To prove (\ref{thetak}), consider multi-indices
$\beta$ and $\gamma$  such that $\abs{\beta},\abs{\gamma} \leq M$. Since
\begin{align*}
& \theta_k(x,\xi,\eta)=
 \int_{\re^{2n}}\ei{x}{(y+z)} \psi_k(\xi,\eta)\left(\sym - \sigma(x,\xi+y,\eta+z)\right)  \hat \phi(y) \hat
\phi(z) \dy \dz,
\end{align*}
we have
\begin{align*}
 &\left( \partial_\xi^\beta \partial_\eta^\gamma \theta_k\right) \xxe
  =\sum_{\lambda \leq \gamma, \omega \leq \beta}
 C_{\beta,\gamma, \omega,\lambda} \left(
\partial_\xi^{\beta-\omega}\partial_\eta^{\gamma-\lambda} \psi\right)(d 2^{-k} \xi,d 2^{-k} \eta)
 (2^{-k} d)^{\abs{\gamma-\lambda} + \abs{\beta-\omega} }\\& \times\int_{\re^{2n}} \hphi(y)\hphi(z)\ei{x}{(y+z)}  \left(
(\partial_\xi^\omega
\partial_\eta^\lambda \sigma)\xxe -  (\partial_\xi^\omega
\partial_\eta^\lambda \sigma) (x,\xi+y,\eta+z)\right)\dy\dz.
\end{align*}
The mean value theorem gives
\begin{align*}
& (\partial_\xi^\omega
\partial_\eta^\lambda \sigma )\xxe -  (\partial_\xi^\omega
\partial_\eta^\lambda  \sigma) (x,\xi+y,\eta+z)
= (\nabla_\xi \partial_\xi^\omega \nabla_\eta \partial_\eta^\lambda \sigma) (x,\tilde
\xi,\tilde \eta) \cdot (y,z),
\end{align*}
where $(\tilde \xi,\tilde \eta)=(\xi,\eta)+ s \,(y,z)$
for some $s \in (0,1)$. Since $\sigma \in \bs$, for $(\xi,\eta)\in \supp(\psi_k)\cap\supp(\theta)$ and $y,\,z\in\supp(\hat{\phi}),$ we then have
\begin{align*}
 \left| (\partial_\xi^\omega
\partial_\eta^\lambda  \sigma)\xxe  \right. & - \left. (\partial_\xi^\omega
\partial_\eta^\lambda  \sigma) (x,\xi+y,\eta+z) \right| \\
& \lesssim \norm{\sigma}{0,M+1}(1+\abs{\tilde \xi}+\abs{\tilde \eta})^{m -
\rho(\abs{\omega}+\abs{\lambda}+1)} |(y,z)|\\
& \lesssim \norm{\sigma}{0,M+1}(1+\abs{ \xi}+\abs{\eta})^{m -
\rho(\abs{\omega}+\abs{\lambda}+1)}  |(y,z)|,
\end{align*}
where we have used that $ \abs{\tilde \xi}+\abs{\tilde \eta} \simeq \abs{\xi}+\abs{\eta}$, since
$\abs{\xi}+\abs{\eta} \simeq 2^kd^{-1}$ and
$\abs{y}+ \abs{z} \leq  d^{-\rho}/4 \leq d^{-1}/4.$
Putting all together, and using again that  $2^k d^{-1}\ge 1,$ $d\le 1,$ and $1 + \abs{\xi}+\abs{\eta} \simeq \abs{\xi}+\abs{\eta} \simeq 2^k d^{-1},$
\begin{align*}
 \abs{\partial_\xi^\beta \partial_\eta^\gamma \theta_k \xxe} \lesssim &\norm{\sigma}{0,M+1}
 \upxe^{m -\rho(\abs{\gamma}+\abs{\beta})} (1+2^k d^{-1})^{-\rho}\\ &\times \mathop{\sum}_{\lambda \leq \gamma,\, \omega \leq \beta}
(2^{-k}d)^{(1-\rho)(\abs{\gamma-\lambda}+\abs{\beta-\omega})}\\
\lesssim & \norm{\sigma}{0,M+1}
 \upxe^{m -\rho(\abs{\gamma}+\abs{\beta})} 2^{-\rho k},
\end{align*}
which gives \eqref{thetak}.
\end{proof}

If $d > 1$ then we split
$\theta$ as
\[
\theta = \theta_1 + \theta_2,
\]
where $\supp(\theta_1) \subset \{ (\xi,\eta): \xpe \leq 2 \}$ (note that in the case
$d > 1$ we also have $\abs{y},\abs{z} \leq d^{-\rho}/8 \leq 1/8$) and
$\supp(\theta_2) \subset \{ (\xi,\eta): \xpe \geq 1 \}$.
A similar reasoning as above shows that $\|\theta_1\|_{0,M}\lesssim\|\sigma\|_{0,M+1}.$
 We then apply Lemma
\ref{compactsupp} to the bilinear pseudo-differential operator with symbol $\theta_1$
and reduce the analysis of $\theta_2$ to the case $d=1$.

\section{Weighted results}\label{sec:weights}

Given a weight $w$ defined on $\re^n$ and $p > 0$, the notation
$L^p_w$ will be used to refer to the weighted Lebesgue space of all
functions $f:\re^n\to\mathbb{C}$ such that
$\norm{f}{L^p_w}:=\int_{\re^n}\abs{f(x)}^pw(x)\,dx<\infty$, when
$w\equiv 1$ we will continue to simply write $L^p$ and
$\norm{f}{L^p},$ respectively.

  If $w_1,\,w_2$ are weights defined on $\re^n,$  $1\le p_1,\,p_{2}<\infty,$ $q>0,$
  and $  w:=w_1^{q/p_1}w_2^{q/p_2},$
we say that $(w_1,w_2)$ satisfies the $A_{(p_1,p_2),q}$ condition
(or that $(w_1,w_2)$ belongs to the bilinear  Muckenhoupt class
$A_{(p_1,p_2),q}$) if
$$[(w_1,w_2)]_{A_{(p_1,p_2),q}}:= \sup_{B }\Big(\frac{1}{|B|}\int_B w(x)\,dx\Big)\,
\prod_{j=1}^2 \Big(\frac{1}{|B|}\int_B w_j(x)^{1-p'_j}\,dx
\Big)^{\frac{q}{p'_j}}<\infty,$$
where the supremum is taken over all Euclidean balls $B\subset
\re^n;$ when $p_j=1$  $\Big(\frac{1}{|B|}\int_B w_j(x)^{1-p'_j}\,dx
\Big)^{\frac{1}{p'_j}}$ is understood as $(\inf_B w_j)^{-1}$.

The  classes $A_{(p_1,p_2),q}$ are  inspired in the classes of
weights $A_{p,q},$ $1\le p,\,q<\infty,$ defined by Muckenhoupt and
Wheeden in \cite{MW} to study weighted norm inequalities for the
fractional integral: a weight $u$ defined on $\re^n$ is in the class
$A_{p,q}$ if
$$\sup_B \left(\frac{1}{|B|}\int_{B} u^{\frac{q}{p}}\,dx\right)\left(\frac{1}{|B|}\int_B u^{(1-p')}\,dx\right)^{\frac{q}{p'}}<\infty.
$$
The classes $A_{(p_1,p_2),p}$\footnote{These classes were denoted by $A_{\vec P}$ in \cite{LOPTT}, with $\vec P = (p_1,p_2)$ determining $1/p=1/p_1+1/p_2$.} for $1/p=1/p_1+1/p_2$  were introduced by  Lerner et al
in \cite{LOPTT} to study characterizations of weights for
boundedness properties  of certain bilinear maximal functions and
bilinear Calder\'on-Zygmund operators in weighted Lebesgue spaces.
 Likewise, as shown by Moen \cite{Moen09}, the classes $A_{(p_1,p_2),q}$ characterize
the weights rendering analogous bounds for bilinear fractional
integral operators .




 Theorem~\ref{thm:bscz} and \cite[Corollary 3.9]{LOPTT} imply the following result.

 \begin{corollary}
 Let $0 \leq \delta \le \rho \le 1,$ $ \delta <
1,$ $0<\rho,$   $m_{cz}=2n(\rho-1)$, $1\le p_1,\,p_2<\infty$ and $p$ given by $\frac{1}{p}=\frac{1}{p_1}+\frac{1}{p_2}$. Suppose $\sigma \in
BS^{m}_{\rho,\delta},$ $m<m_{cz},$  $(w_1,w_2)$ satisfies the $A_{(p_1,p_2),p}$ condition and $ w= w_1^{p/p_1}w_2^{p/p_2}.$
\begin{enumerate}[(a)]
\item If $1<p_1,p_2<\infty$ then there exists $K,N\in\na_0$ such that
\[
\|T_\sigma(f,g)\|_{L^p_w} \lesssim \norm{\sigma}{K,N}\,\|f\|_{L^{p_1}_{w_1}}\|g\|_{L^{p_2}_{w_2}}.
\]
\item If $1\le p_1,p_2<\infty$ and $p_1=1$ or $p_2=1$ then there exists $K,N\in\na_0$ such that
\[
\|T(f,g)\|_{L^{p,\infty}_{w}}\ \lesssim \norm{\sigma}{K,N}\, \|f\|_{L^{p_1}_{w_1}}\|g\|_{L^{p_2}_{w_2}}.
\]
\end{enumerate}

\end{corollary}

Inequality \eqref{eqn:boundbyfracint} and \cite[Theorem 3.5]{Moen09} yield the following:
\begin{corollary}[Weighted version of Theorem~\ref{thm:bssobolev}]
Let $0 \leq \delta \leq 1,$ $0 < \rho \leq 1$, $s \in (0,2n)$, and
$m_{s}:= 2n(\rho -1) - \rho s$. If $\sigma \in
BS^{m}_{\rho,\delta},$ $m\le m_s,$ and $\frac{1}{q} = \frac{1}{p_1}
+ \frac{1}{p_1} - \frac{s}{n},$ $1 < p_1, p_2 < \infty$, then there
exist nonnegative integers $K$ and $N$ such that
\[
\norm{T_\sigma(f,g)}{L^q_w}\lesssim \norm{\sigma}{K,N}\,\norm{f}{L^{p_1}_{w_1}}\norm{g}{L^{p_2}_{w_2}},
\]
 for $ w:= w_1^{q/p_1}w_2^{q/p_2}$ and  pairs of weights $(w_1,w_2)$ satisfying the $A_{(p_1,p_2),q}$ condition.
\end{corollary}

\end{document}